\documentclass[a4paper,12pt]{article}

\usepackage{color,makeidx,bm,latexsym,amscd,amsmath,amssymb,enumerate,stmaryrd,amsthm,float,graphicx,psfrag,geometry}

\usepackage[all]{xy}
\usepackage{bm}

\usepackage{tikz,epic,eso-pic}

\theoremstyle{plain}
\newtheorem{theorem}{Theorem}[section]

\theoremstyle{definition}
\newtheorem{definition}[theorem]{Definition}
\newtheorem{remark}[theorem]{Remark}
\newtheorem{example}[theorem]{Example}
\newtheorem{proposition}[theorem]{Proposition}

\newcommand{\C}{\mathbb{C}}

\newcommand{\R}{\mathbb{R}}

\newcommand{\A}{\mathcal{A}}

\newcommand{\ch}{\operatorname{ch}} 
\newcommand{\Rect}{\operatorname{Rect}} 
\newcommand{\FS}{\operatorname{\mathbb{S}}} 

\newcommand{\disk}{\operatorname{\mathtt{D}^4}}
\newcommand{\sphere}{\operatorname{\mathtt{S}^3}}

\title{Divides with cusps and Kirby diagrams for line arrangements} 
\author{Sakumi Sugawara\thanks{Department of Mathematics, Graduate School of Science, 
Hokkaido University, North 10, West 8, Kita-ku, 
Sapporo 060-0810, JAPAN 
E-mail: sugawara.sakumi.f5@elms.hokudai.ac.jp}, 
Masahiko Yoshinaga\thanks{Department of Mathematics, Faculty of Science, 
Hokkaido University, North 10, West 8, Kita-ku, 
Sapporo 060-0810, JAPAN 
E-mail: yoshinaga@math.sci.hokudai.ac.jp}\\
\ \\
\emph{Dedicated to the memory of Prof. Stefan Papadima}}
\date{\today}
\begin{document}
\maketitle

\begin{abstract}
The complement of a complexified real line arrangement is an 
affine surface. 
It is classically known that such a space 
has a handle decomposition up to $2$-handles. 
We will describe the handle decomposition induced from Lefschetz 
hyperplane section theorem for such a space. 
To describe the Kirby diagram, we introduce the 
notion of the divide with cusps which is a generalization of the divide 
introduced by A'Campo. 
\textbf{MSC classes}:	32S50, 32Q55, 52C35
\end{abstract}


\section{Introduction}
\label{sec:intro}

Let $\A=\{H_1, \dots, H_n\}$ be a set of hyperplanes in a complex vector 
space $\C^\ell$. 
The complement $M(\A)=\C^\ell\smallsetminus\bigcup_{H\in\A}H$ is 
an important topological space that has been studied from the topological and 
combinatorial perspective (\cite{orl-ter}). 
One of the most peculiar properties of $M(\A)$ is 
the so-called \emph{minimality}, that is, $M(\A)$ is homotopy equivalent to a 
finite CW complex such that the number of $k$-cells is equal to 
the $k$-th Betti number $b_k(M(\A))$ for all $k\geq 0$ \cite{ran-mor, dim-pap} 
(see also \cite{fal-hom} for the case $\ell=2$). 
Note that the original proof did not give a precise description of attaching maps 
of minimal CW complexes. Later several approaches to this problem 
were developed for complexified real arrangements. 
For example, descriptions via Lefschetz hyperplane section theorem 
\cite{yos-lef} and via discrete Morse theory \cite{sal-sett} are known. 

When $\ell=2$, the complement $M(\A)$ is a complex smooth affine 
surface. It has been well known that such a space is a $4$-manifold 
which has a handle decomposition by $0$-, $1$-, and $2$-handles. 
Such handle decompositions are completely described by 
Kirby diagrams (\cite{akb, kir-top, gom-sti}). 
However, as far as the authors know, 
there are no results available in the literature concerning explicit 
handle decompositions and Kirby diagrams for $M(\A)$. 
(Recently, the handle decompositions for the complement of 
real $1$-dimensional subspaces in $\R^\ell$ were obtained 
\cite{ish-oya}.) The purpose of the present paper is to 
give a description of handle decompositions and Kirby diagrams 
of $M(\A)$ for complexified real line arrangements. 

Our approach consists of two steps. The first step is describing 
the handle decomposition explicitly (\S \ref{sec:handle}). 
This part is a refinement of a previous description (\cite{yos-lef, yos-str}) 
of the attaching maps of $2$-cells on $1$-skeletons. We will 
describe the attaching of $2$-handles using piecewise linear maps. 

The second step is describing the Kirby diagram explicitly. 
For this purpose, we introduce the notion of the 
\emph{divide with cusps}. 

Simply speaking, a \emph{divide} $C$ is a union of immersed curves 
in the $2$-dimensional unit disk $D^2$. 
Although the divide was originally introduced in the context of 
complex plane curve singularities, 
it is of interest in low dimensional topology 
since one can associate a link $L(C)$ in $S^3$ to a divide $C$ (\cite{acampo}). 
The notion of the divide can be considered as 
a ``$2$-dimensional way'' of describing certain links. 
There are also several generalizations of the notion of divide, e.g., 
oriented divides and free divides (\cite{gi-top, gi-osaka}), etc. 
(In particular, the notion of cusp in the present paper 
may be considered as ``divisors'' (in the sense of 
\cite{gi-osaka}) equipped with half-integer signs $\pm\frac{1}{2}$.) 

The \emph{divide with cusps} is, roughly speaking, a generalization 
that allows cusps in the union of curves $C$ in the $D^2$. Since 
the tangent space is well-defined also at a cusp, one can associate 
a link as in the case of divides. Our main result asserts that 
there exists a divide with cusps such that its link represents 
the Kirby diagram of $M(\A)$. 

The paper is organized as follows. 
In \S \ref{sec:sphere} we prepare piecewise linear descriptions 
of the $3$-sphere $S^3$ and certain circles in it which will be used later. 
In \S \ref{sec:divides} we define divides with cusps and associated 
links. 
In \S \ref{sec:handle} we give a piecewise linear description of 
handle decompositions for the complements to complexified 
real line arrangements. Here Kirby diagrams (the framed links consisting of 
boundaries of carved disks and 
attaching circles of $2$-handles) are described in terms 
of piecewise linear curves introduced in the previous section. 
In \S \ref{sec:kirby} we describe the Kirby diagram using 
divide with cusps.















\section{Standard and PL models for $S^3$}

\label{sec:sphere}

Let $D^2\subset \R^2$ be the $2$-disk. The starting point of 
the theory of divides (\cite{acampo}) is to describe 
the $3$-sphere $S^3$ in terms of tangent vectors on $D^2$ as follows. 
\[
\{(\bm{x}, \bm{y})\mid \bm{x}\in D^2, \bm{y}\in T_{\bm{x}}\R^2, 
|\bm{x}|^2+|\bm{y}|^2=1\}, 
\]
where $\bm{x}=(x_1, x_2), \bm{y}=(y_1, y_2),$ and $|\bm{x}|=\sqrt{x_1^2+x_2^2}$. 
We will also identify $T\R^2$ with $\C^2$ by 
$(\bm{x}, \bm{y})\longmapsto \bm{x}+\sqrt{-1}\bm{y}$. Therefore, we 
sometimes call $\bm{x}$ (resp. $\bm{y}$) the \emph{real part} 
(resp. \emph{imaginary part}) of $(\bm{x}, \bm{y})$. 

In this paper, we will use the following piecewise linear (rectangular) 
model for $S^3$. 
For $\bm{x}=(x_1, x_2)$, 
let us denote the maximum norm by $||\bm{x}||_\infty:=\max\{|x_1|, |x_2|\}$. 
Let $R>0$ be a positive real number. The rectangle $\Rect(R, 1)$ is 
defined as 
\[
\begin{split}
\Rect(R, 1):=&[-R, R]\times[-1, 1]\\
=&\{\bm{x}=(x_1, x_2)\in\R^2\mid 
-R\leq x_1\leq R, -1\leq x_2\leq 1\}. 
\end{split}
\]
For $\bm{x}\in\Rect(R, 1)$, let $\delta(\bm{x})$ be 
the distance between $\bm{x}$ and the boundary $\partial\Rect(R, 1)$ 
with respect to the norm $||\cdot||_\infty$. In other words, 
$\delta(\bm{x})=\min\{x_1+R, R-x_1, x_2+1, 1-x_2\}$. 
We will use the following piecewise linear presentation $\sphere (R, 1)$ of the 
sphere (see Figure \ref{fig:rect}). 
\begin{equation}
\label{eq:S3}
S^3\approx \sphere (R, 1):=
\{(\bm{x}, \bm{y})\mid \bm{x}\in \Rect(R, 1), \bm{y}\in T_{\bm{x}}\R^2, 
||\bm{y}||_\infty=\delta(\bm{x})\}. 
\end{equation}
\begin{figure}[htbp]
\centering
\begin{tikzpicture}
 \draw [thick](-5,-1) rectangle (5,1);


\draw (-5,1) node[left] {$(-5,1)$}; 
\draw (-5,-1) node[left] {$(-5,-1)$}; 
\draw (5,-1) node[right] {$(5,-1)$}; 
\draw (5,1) node[right] {$(5,1)$}; 
 
 \coordinate [label =below left:$O$] (O) at (0,0);
 \fill [black] (O) circle (0.06);
 \draw (1.3,-0.4) rectangle (2.7,1);
 \draw [->,>=stealth] (2, 0.3)-- +(0.7, 0);
 \draw [->,>=stealth] (2, 0.3)-- +(0.7, 0.4);
 \draw [->,>=stealth] (2, 0.3)-- +(0.7, 0.7);
 \draw [->,>=stealth] (2, 0.3)-- +(0.4, 0.7);
 \draw [->,>=stealth] (2, 0.3)-- +(0, 0.7);
 \draw [->,>=stealth] (2, 0.3)-- +(-0.4, 0.7);
 \draw [->,>=stealth] (2, 0.3)-- +(-0.7, 0.7);
 \draw [->,>=stealth] (2, 0.3)-- +(-0.7, 0.4);
 \draw [->,>=stealth] (2, 0.3)-- +(-0.7, 0);
 \draw [->,>=stealth] (2, 0.3)-- +(-0.7, -0.4);
 \draw [->,>=stealth] (2, 0.3)-- +(-0.7, -0.7);
 \draw [->,>=stealth] (2, 0.3)-- +(-0.4, -0.7);
 \draw [->,>=stealth] (2, 0.3)-- +(0, -0.7);
 \draw [->,>=stealth] (2, 0.3)-- +(0.4, -0.7);
 \draw [->,>=stealth] (2, 0.3)-- +(0.7, -0.7);
 \draw [->,>=stealth] (2, 0.3)-- +(0.7, -0.4);
\end{tikzpicture}
\caption{$\Rect(5, 1)$ and a circle $\FS(\bm{a}, \bm{a})$ with 
$\bm{a}=(2, 0.3)$}\label{fig:rect}
\end{figure}
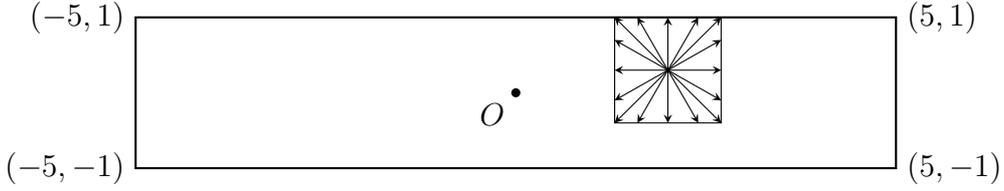
Next, we define a special family of circles which are defined as 
images of piecewise linear maps from the boundary of the 
square $\partial\Rect(1, 1)$. 
Let $\bm{a}_1=(a_1, b), \bm{a}_2=(a_2, b)\in\Rect(R, 1)$ be two points with 
the same height $b$. 
For simplicity, we assume $a_1\leq a_2, 0\leq |a_1|, |a_2|\ll R$, and $0\leq b<1$. 
We decompose the boundary $\partial\Rect(1, 1)$ into four edges 
$R_1, R_2, R_3, R_4$ as shown in Figure \ref{fig:r1234}. 
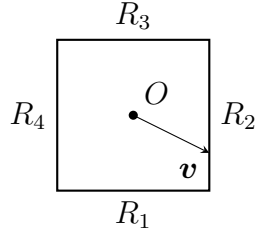
\begin{figure}[htbp]
\centering
\begin{tikzpicture}
 \draw [thick](-1,-1) rectangle (1,1);
 \coordinate [label = below :$R_1$] (A) at (0,-1);
 \coordinate [label = right :$R_2$] (B) at (1,0);
 \coordinate [label = above :$R_3$] (C) at (0,1);
 \coordinate [label = left :$R_4$] (D) at (-1,0);
 
 \coordinate [label =above right:$O$] (O) at (0,0);
 \fill [black] (O) circle (0.06);

 \draw [->,>=stealth] (O)--(1,-0.5) node [below left] {$\bm{v}$};
\end{tikzpicture}
\caption{decomposition into $R_1$, ... ,  $R_4$}\label{fig:r1234}

\end{figure}

Let $\bm{v}=(v_1, v_2)\in\partial \Rect(1, 1)=R_1\cup R_2\cup R_3\cup R_4$. 
Define the map $F:\partial \Rect(1, 1)\longrightarrow \sphere(R, 1)$ 
as follows. 
\[
F(\bm{v})=
\begin{cases}
\left(\left(\frac{a_2-a_1}{2}v_1+\frac{a_1+a_2}{2}, b\right), (1-b)\bm{v}\right), & v_2=-1\ (R_1),\\
\left(\left(\frac{a_1-a_2}{2}v_2+\frac{a_1+a_2}{2}, b\right), (1-b)\bm{v}\right), & v_1=1\ (R_2),\\
\left(\left(\frac{a_1-a_2}{2}v_1+\frac{a_1+a_2}{2}, b\right), (1-b)\bm{v}\right), &v_2=1\ (R_3),\\
\left(\left(\frac{a_2-a_1}{2}v_2+\frac{a_1+a_2}{2}, b\right), (1-b)\bm{v}\right), & v_1=-1\ (R_4), 
\end{cases}
\]
(See also Figure \ref{fig:FS}). We denote the image of this map by 
\[
\FS(\bm{a}_1, \bm{a}_2)=F(R_1)\cup F(R_2)\cup F(R_3)\cup F(R_4). 
\] 
Since the argument of the tangent vector is preserved, 
$\FS(\bm{a}_1, \bm{a}_2)$ is an embedded (piecewise linear) circle. 
When $\bm{a}_1=\bm{a}_2$, $\FS(\bm{a}, \bm{a})$ is a circle (square) as in 
Figure \ref{fig:rect}. 

\begin{figure}[htbp]
\centering
\begin{tikzpicture}

 \coordinate (O1) at (0,0);
 \draw [thick] (O1)+(-3,-1) rectangle +(3,1);
\draw[thick] (O1) +(-1,0.3) -- (2,0.3);
\fill[black] (O1)+(-1,0.3) node [above] {\scriptsize $\bm{a}_1=F(-1, -1)$} circle(0.06); 
\fill[black] (O1)+(2,0.3) node [above] {\scriptsize $\bm{a}_2=F(1, -1)$} circle(0.06); 
\draw [->,>=stealth, thick] (O1)+(-1,0.3)--(-1.7,-0.4);
\draw [->,>=stealth, thick] (O1)+(-0.25,0.3)--(-0.75,-0.4);
\draw [->,>=stealth, thick] (O1)+(0.5,0.3)--(0.5,-0.4);
\draw [->,>=stealth, thick] (O1)+(1.25,0.3)--(1.75,-0.4);
\draw [->,>=stealth, thick] (O1)+(2,0.3)--(2.7,-0.4);
\draw (O1)+(0, -1) node [above] {$F(R_1)$};

 \coordinate (O2) at (4,2.5);
 \draw [thick] (O2)+(-3,-1) rectangle +(3,1);
\draw[thick] (O2) +(-1,0.3) -- ++(2,0.3);
\fill[black] (O2)+(-1,0.3) node [below] {\scriptsize $F(1, 1)$} circle(0.06); 
\fill[black] (O2)+(2,0.3) node [above] {\scriptsize $F(1, -1)$} circle(0.06); 
\draw [->,>=stealth, thick] (O2)+(-1,0.3)-- ++(-0.3,1);
\draw [->,>=stealth, thick] (O2)+(-0.25,0.3)-- ++(0.45,0.8);
\draw [->,>=stealth, thick] (O2)+(0.5,0.35)--++(1.2,0.35);
\draw [->,>=stealth, thick] (O2)+(1.25,0.3)--++(1.95,-0.2);
\draw [->,>=stealth, thick] (O2)+(2,0.3)--++(2.7,-0.4);
\draw (O2)+(0, -1) node [above] {$F(R_2)$};

 \coordinate (O3) at (0,5);
 \draw [thick] (O3)+(-3,-1) rectangle +(3,1);
\draw[thick] (O3) +(-1,0.3) -- ++(2,0.3);
\fill[black] (O3)+(-1,0.3) node [below] {\scriptsize $F(1, 1)$} circle(0.06); 
\fill[black] (O3)+(2,0.3) node [below] {\scriptsize $F(-1, 1)$} circle(0.06); 
\draw [->,>=stealth, thick] (O3)+(-1,0.3)-- ++(-0.3,1);
\draw [->,>=stealth, thick] (O3)+(-0.25,0.3)-- ++(0.25,1);
\draw [->,>=stealth, thick] (O3)+(0.5,0.35)--++(0.5,1);
\draw [->,>=stealth, thick] (O3)+(1.25,0.3)--++(0.75,1);
\draw [->,>=stealth, thick] (O3)+(2,0.3)--++(1.3,1);
\draw (O3)+(0, -1) node [above] {$F(R_3)$};

 \coordinate (O4) at (-4,2.5);
 \draw [thick] (O4)+(-3,-1) rectangle +(3,1);
\draw[thick] (O4) +(-1,0.3) -- ++(2,0.3);
\fill[black] (O4)+(-1,0.3) node [above] {\scriptsize $F(-1, -1)$} circle(0.06); 
\fill[black] (O4)+(2,0.3) node [below] {\scriptsize $F(-1, 1)$} circle(0.06); 
\draw [->,>=stealth, thick] (O4)+(-1,0.3)-- ++(-1.7,-0.4);
\draw [->,>=stealth, thick] (O4)+(-0.25,0.3)-- ++(-0.95,-0.2);
\draw [->,>=stealth, thick] (O4)+(0.5,0.35)--++(-0.2,0.35);
\draw [->,>=stealth, thick] (O4)+(1.25,0.3)--++(0.55,0.5);
\draw [->,>=stealth, thick] (O4)+(2,0.3)--++(1.3,1);
\draw (O4)+(0, -1) node [above] {$F(R_4)$};




\end{tikzpicture}
\caption{$F(R_1)$, $F(R_2)$, $F(R_3)$, and $F(R_4)$}\label{fig:FS}

\end{figure}
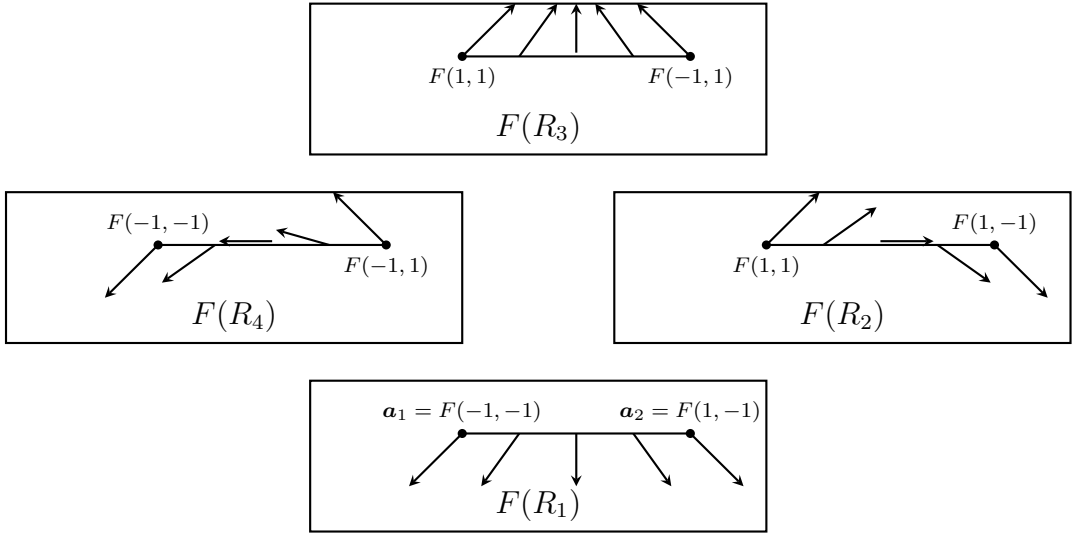
For simplicity, we draw $\FS(\bm{a}_1, \bm{a}_2)$ as in Figure \ref{fig:shortFS}. 
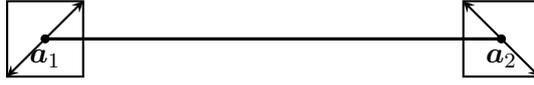
\begin{figure}[htbp]
\centering
\begin{tikzpicture}
\draw [very thick] (-3,0)--(3,0);
\fill[black] (-3,0) node [below] {\small $\bm{a}_1$} circle(0.06);
\draw [->,>=stealth, thick] (-3,0)--++(0.5,0.5);
\draw [->,>=stealth, thick] (-3,0)--++(-0.5,-0.5);
\draw [thick] (-3.5,-0.5) rectangle +(1,1);

\fill[black] (3,0) node [below] {\small $\bm{a}_2$} circle(0.06);
\draw [->,>=stealth, thick] (3,0)--++(-0.5,0.5);
\draw [->,>=stealth, thick] (3,0)--++(0.5,-0.5);
\draw [thick] (2.5,-0.5) rectangle +(1,1);

\end{tikzpicture}
\caption{Pictorial notation for $\FS(\bm{a}_1, \bm{a}_2)=
F(R_1)\cup F(R_2)\cup F(R_3)\cup F(R_4)$}\label{fig:shortFS}
\end{figure}
\begin{proposition}
\label{prop:trivial}
$\FS(\bm{a}_1, \bm{a}_2)$ is isotopic to $ \FS(\bm{a}_1, \bm{a}_1)$. 
Furthermore, $\FS(\bm{a}_1, \bm{a}_2)$ is a trivial knot. 
\end{proposition}

\begin{proof}
It is easily seen that 
$\FS(\bm{a}_1, t\bm{a}_1+(1-t)\bm{a}_2)$, $0\leq t\leq 1$, gives 
an isotopy and $\FS(\bm{a}_1, \bm{a}_1)\subset \sphere(R, 1)$ is trivial. 
\end{proof}

\begin{remark}
Note that $\FS(\bm{a}_1, \bm{a}_2)$ can be regarded as a piecewise 
linear knot with four edges. 
The triviality of $\FS(\bm{a}_1, \bm{a}_2)$ may also be 
justified by the fact that every knotted piecewise linear knot 
has at least six edges (\cite{cal-mil}). 
\end{remark}

\begin{proposition}
\label{prop:linking}
Let $\bm{a}_i=(a_i, b), \bm{a}_i'=(a_i, b')$, ($i=1, 2$). Assume $b\neq b'$. 
Then $\FS(\bm{a}_1, \bm{a}_2)\cap\FS(\bm{a}_1', \bm{a}_2')=\emptyset$.  
Furthermore, 
$\FS(\bm{a}_1, \bm{a}_2)$ and $\FS(\bm{a}_1', \bm{a}_2')$ are not linked. 
In particular, the linking number is 
$\operatorname{Lk}(\FS(\bm{a}_1, \bm{a}_2), \FS(\bm{a}_1', \bm{a}_2'))=0$. 
\end{proposition}

\begin{proof}
Since they have no common real parts, 
$\FS(\bm{a}_1, \bm{a}_2)\cap\FS(\bm{a}_1', \bm{a}_2')=\emptyset$ is 
obvious. 
We may assume $b>b'$. Let $L:=\{x_2=\frac{b+b'}{2}\}$ be the horizontal 
line separating two segments $[\bm{a}_1, \bm{a}_2]$ and 
$[\bm{a}_1', \bm{a}_2']$. Then 
\[
\{(\bm{x}, \bm{y})\mid \bm{x}\in L, \bm{y}\in T_{\bm{x}}\R^2, 
||\bm{y}||_\infty=\delta(\bm{x})\}
\]
defines an $S^2$ in $\sphere(R, 1)$ which separates 
$\FS(\bm{a}_1, \bm{a}_2)$ and $\FS(\bm{a}_1', \bm{a}_2')$ (Figure \ref{fig:separation}). 
Therefore, they are not linked. 
\end{proof}

\begin{figure}[htbp]
\centering
\begin{tikzpicture}

\draw[thick] (-5,0.2) node [above right] {$L$} --(5, 0.2) ;

\draw [very thick] (-2,1)--(3,1);
\fill[black] (-2,1) node [below] {\small $\bm{a}_1$} circle(0.06);
\draw [->,>=stealth] (-2,1)--++(0.5,0.5);
\draw [->,>=stealth] (-2,1)--++(-0.5,-0.5);
\draw [thick] (-2.5,0.5) rectangle +(1,1);

\fill[black] (3,1) node [below] {\small $\bm{a}_2$} circle(0.06);
\draw [->,>=stealth] (3,1)--++(-0.5,0.5);
\draw [->,>=stealth] (3,1)--++(0.5,-0.5);
\draw [thick] (2.5,0.5) rectangle +(1,1);

\draw [very thick] (-3,-1)--(2,-1);
\fill[black] (-3,-1) node [below] {\small $\bm{a}_1'$} circle(0.06);
\draw [->,>=stealth] (-3,-1)--++(0.5,0.5);
\draw [->,>=stealth] (-3,-1)--++(-0.5,-0.5);
\draw [thick] (-3.5,-1.5) rectangle +(1,1);

\fill[black] (2,-1) node [below] {\small $\bm{a}_2'$} circle(0.06);
\draw [->,>=stealth] (2,-1)--++(-0.5,0.5);
\draw [->,>=stealth] (2,-1)--++(0.5,-0.5);
\draw [thick] (1.5,-1.5) rectangle +(1,1);

\end{tikzpicture}
\caption{Separation of $\FS(\bm{a}_1, \bm{a_2})$ and 
$\FS(\bm{a}_1', \bm{a_2}')$}\label{fig:separation}

\end{figure}
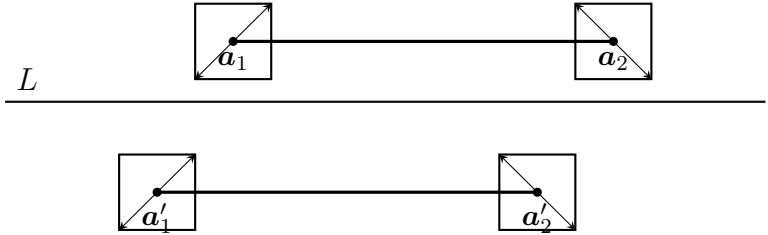

\section{Divides with cusps}

\label{sec:divides}

In this section, we introduce the notion of the divide with cusps and 
associated links. 

Let $X$ be a compact $1$-dimensional manifold. 
Then, 
$X$ is expressed as a disjoint union of finitely many intervals and circles 
$\bigsqcup_j I_j \sqcup\bigsqcup_k S_k$, where $I_j$ and $S_k$ are 
intervals and circles, respectively. 
We have $\partial X=\bigsqcup_j\partial I_j$. 
\begin{definition}
A \emph{divide with cusps} $C$ is the image of a continuous map 
$\alpha: (X, \partial X)\longrightarrow (D^2, \partial D^2)$ 
satisfying the following conditions. 
\begin{itemize}
\item[(i)] 
$\alpha^{-1}(\partial D)=\partial X$, $\alpha|_{\partial X}$ is injective, 
and $C$ is transversal to $\partial D$. 
\item[(ii)] 
There are finitely many points $p_1, \dots, p_s$ in the interior of $X$ 
such that 
\begin{itemize}
\item 
$\alpha$ is an immersion on $X\smallsetminus\{p_1, \dots, p_s\}$ 
and $\alpha(p_i)$ is a cusp of $C$. 
\item 
$\alpha^{-1}(\alpha(p_i))=\{p_i\}$. 
\end{itemize}
\item[(iii)] 
Except for cusps and double points, there are no singularities. 
\end{itemize}
(See Figure \ref{fig:cuspdivide}.) 
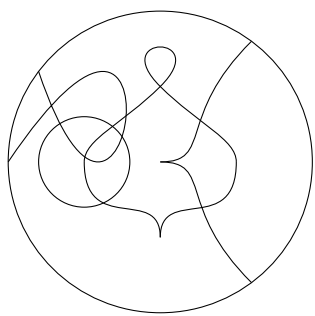
\begin{figure}[htbp]
\centering
\begin{tikzpicture}

\draw (0,0) circle (2);
\draw (-1,0) .. controls ++(0,0.3) and ++(-0.5, -0.5) .. (0,1);
\draw (1,0) .. controls ++(0,0.3) and ++(0.5, -0.5) .. (0,1);
\draw (0,1) .. controls ++(0.7,0.7) and ++(-0.7,0.7) .. (0,1);
\draw (-1,0) .. controls ++(0,-1) and ++(0,0.7) .. (0,-1);
\draw (1,0) .. controls ++(0,-1) and ++(0,0.7) .. (0,-1);

\draw (1.2,1.6) .. controls ++(-1,-1) and ++(0.7,0) .. (0,0);
\draw (1.2,-1.6) .. controls ++(-1,1) and ++(0.7,0) .. (0,0);

\draw (-1.6,1.2) .. controls ++(1.2,-3.6) and ++(2.4,3.6) .. (-2,0);

\draw (-1,0) circle (0.6);

\end{tikzpicture}
\caption{A divide with cusps}\label{fig:cuspdivide}
\end{figure}
\end{definition}
Let $C$ be a divide with cusps. We consider tangent space at 
$\bm{x}=(x_1, x_2)\in C$. If $\bm{x}$ is neither double points nor 
cusps, then $T_{\bm{x}}C$ is as usual. Note that there is a 
natural tangent space $T_{\bm{x}}C$ at a cusp (Figure \ref{fig:tangent}). 
At a double point, $T_{\bm{x}}C$ is a union of two tangent spaces. 


\begin{figure}[htbp]
\centering
\begin{tikzpicture}

\draw (0,0) .. controls (0,1) and (4,1) .. (4,0) node [right] {$C$};
\draw (0,0) .. controls (0,-1) and (2,0) .. (2,-1);
\draw (4,0) .. controls (4,-1) and (2,0) .. (2,-1);

\fill[black] (2,-1) node [left] {\small $p$} circle(0.06);

\draw [dashed] (2,0) -- (2,-2) node [right] {\small $T_p C$};

\end{tikzpicture}
\caption{Tangent space $T_pC$ at cusp $p\in C$}\label{fig:tangent}
\end{figure}
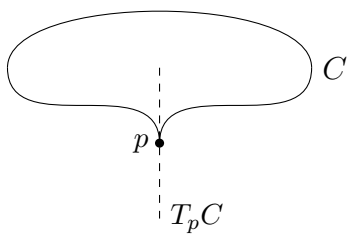
To a divide with cusps $C$, the link $L(C)$ is defined as 
\[
L(C)=\{(\bm{x}, \bm{y}), \mid \bm{x}\in C, \bm{y}\in T_{\bm{x}}(C), 
|\bm{x}|^2+|\bm{y}|^2=1\}\subset S^3. 
\]
By definition, when $\bm{x}\in C$ is in the interior of $D^2$, 
there are exactly two tangent vectors $\pm\bm{y}\in T_{\bm{x}}C$ which 
is contained in $L(C)$. The vector $\bm{y}$ is determined by the argument. 
Hence (over the interior) the link $L(C)$ can be drawn in $C\times S^1$. 
Note that 
two kinds of cusps are corresponding to 
positive and negative half twists (Figure \ref{fig:twist}). 





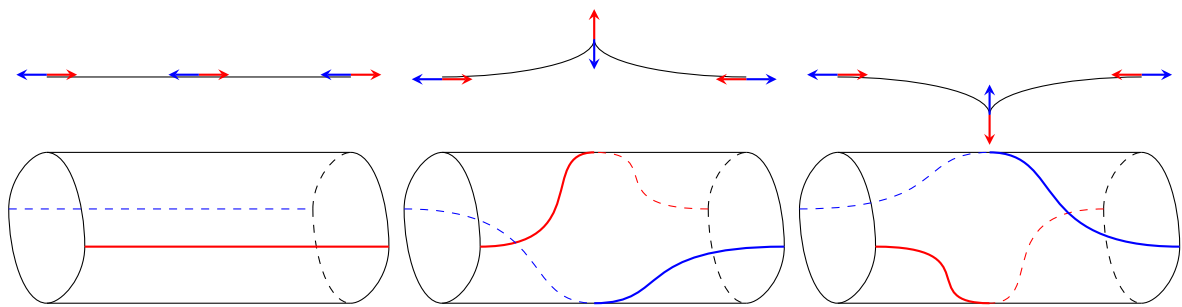
\begin{figure}[htbp]
\centering
\begin{tikzpicture}


\coordinate (O1) at (0,1);
\coordinate (O11) at (4,1);
\coordinate (O2) at (5.2,1);
\coordinate (O22) at (9.2,1);
\coordinate (O3) at (10.4,1);
\coordinate (O33) at (14.4,1);


\draw (O1)+(0,2) -- ++(4,2);
\draw [->,>=stealth,thick, red] (O1)+(0,2.03)--++(0.4,2.03);
\draw [->,>=stealth,thick, blue] (O1)+(0,2.03)--++(-0.4,2.03);
\draw [->,>=stealth,thick, red] (O1)+(2,2.03)--++(2.4,2.03);
\draw [->,>=stealth,thick, blue] (O1)+(2,2.03)--++(1.6,2.03);
\draw [->,>=stealth,thick, red] (O11)+(0,2.03)--++(0.4,2.03);
\draw [->,>=stealth,thick, blue] (O11)+(0,2.03)--++(-0.4,2.03);

\draw (O1)+(-0.5,0.25) .. controls ++(0,0.4) and ++(-0.2,0) .. ++(0,1);
\draw (O1)+(-0.5,0.25) .. controls ++(0,-0.6) and ++(-0.3,0) .. ++(0,-1);
\draw (O1)+(0.5,-0.25) .. controls ++(0,0.6) and ++(0.3,0) .. ++(0,1);
\draw (O1)+(0.5,-0.25) .. controls ++(0,-0.4) and ++(0.2,0) .. ++(0,-1);

\draw [dashed] (O11)+(-0.5,0.25) .. controls ++(0,0.4) and ++(-0.2,0) .. ++(0,1);
\draw [dashed] (O11)+(-0.5,0.25) .. controls ++(0,-0.6) and ++(-0.3,0) .. ++(0,-1);
\draw (O11)+(0.5,-0.25) .. controls ++(0,0.6) and ++(0.3,0) .. ++(0,1);
\draw (O11)+(0.5,-0.25) .. controls ++(0,-0.4) and ++(0.2,0) .. ++(0,-1);

\draw (O1)+(0,1) -- ++(4,1);
\draw (O1)+(0,-1) -- ++(4,-1);

\draw [red, thick] (O1)+(0.5,-0.25) -- ++(4.5,-0.25);
\draw [blue, dashed] (O1)+(-0.5,0.25) -- ++(3.5,0.25);


\draw (O2)+(0,2) .. controls ++(1,0) and ++(0,-0.3) .. ++(2,2.5);
\draw (O22)+(0,2) .. controls ++(-1,0) and ++(0,-0.3) .. ++(-2,2.5);
\draw [->,>=stealth,thick, red] (O2)+(0,1.97)--++(0.4,1.97);
\draw [->,>=stealth,thick, blue] (O2)+(0,1.97)--++(-0.4,1.97);
\draw [->,>=stealth,thick, red] (O2)+(2,2.5)--++(2,2.9);
\draw [->,>=stealth,thick, blue] (O2)+(2,2.5)--++(2,2.1);
\draw [->,>=stealth,thick, blue] (O22)+(0,1.97)--++(0.4,1.97);
\draw [->,>=stealth,thick, red] (O22)+(0,1.97)--++(-0.4,1.97);

\draw (O2)+(-0.5,0.25) .. controls ++(0,0.4) and ++(-0.2,0) .. ++(0,1);
\draw (O2)+(-0.5,0.25) .. controls ++(0,-0.6) and ++(-0.3,0) .. ++(0,-1);
\draw (O2)+(0.5,-0.25) .. controls ++(0,0.6) and ++(0.3,0) .. ++(0,1);
\draw (O2)+(0.5,-0.25) .. controls ++(0,-0.4) and ++(0.2,0) .. ++(0,-1);

\draw [dashed] (O22)+(-0.5,0.25) .. controls ++(0,0.4) and ++(-0.2,0) .. ++(0,1);
\draw [dashed] (O22)+(-0.5,0.25) .. controls ++(0,-0.6) and ++(-0.3,0) .. ++(0,-1);
\draw (O22)+(0.5,-0.25) .. controls ++(0,0.6) and ++(0.3,0) .. ++(0,1);
\draw (O22)+(0.5,-0.25) .. controls ++(0,-0.4) and ++(0.2,0) .. ++(0,-1);

\draw (O2)+(0,1) -- ++(4,1);
\draw (O2)+(0,-1) -- ++(4,-1);

\draw [red, thick] (O2)+(0.5,-0.25) .. controls ++(1.5,0) and ++(-0.75,0) .. ++(2,1);
\draw [red, dashed] (O22)+(-0.5,0.25) .. controls ++(-1.5,0) and ++(1,0) .. ++(-2,1);
\draw [blue, dashed] (O2)+(-0.5,0.25) .. controls ++(2,0) and ++(-1,0) .. ++(2,-1);
\draw [blue, thick] (O22)+(0.5,-0.25) .. controls ++(-2,0) and ++(1,0) .. ++(-2,-1);


\draw (O3)+(0,2) .. controls ++(1,0) and ++(0,0.3) .. ++(2,1.5);
\draw (O33)+(0,2) .. controls ++(-1,0) and ++(0,0.3) .. ++(-2,1.5);
\draw [->,>=stealth,thick, red] (O3)+(0,2.03)--++(0.4,2.03);
\draw [->,>=stealth,thick, blue] (O3)+(0,2.03)--++(-0.4,2.03);
\draw [->,>=stealth,thick, red] (O3)+(2,1.5)--++(2,1.1);
\draw [->,>=stealth,thick, blue] (O3)+(2,1.5)--++(2,1.9);
\draw [->,>=stealth,thick, blue] (O33)+(0,2.03)--++(0.4,2.03);
\draw [->,>=stealth,thick, red] (O33)+(0,2.03)--++(-0.4,2.03);

\draw (O3)+(-0.5,0.25) .. controls ++(0,0.4) and ++(-0.2,0) .. ++(0,1);
\draw (O3)+(-0.5,0.25) .. controls ++(0,-0.6) and ++(-0.3,0) .. ++(0,-1);
\draw (O3)+(0.5,-0.25) .. controls ++(0,0.6) and ++(0.3,0) .. ++(0,1);
\draw (O3)+(0.5,-0.25) .. controls ++(0,-0.4) and ++(0.2,0) .. ++(0,-1);

\draw [dashed] (O33)+(-0.5,0.25) .. controls ++(0,0.4) and ++(-0.2,0) .. ++(0,1);
\draw [dashed] (O33)+(-0.5,0.25) .. controls ++(0,-0.6) and ++(-0.3,0) .. ++(0,-1);
\draw (O33)+(0.5,-0.25) .. controls ++(0,0.6) and ++(0.3,0) .. ++(0,1);
\draw (O33)+(0.5,-0.25) .. controls ++(0,-0.4) and ++(0.2,0) .. ++(0,-1);

\draw (O3)+(0,1) -- ++(4,1);
\draw (O3)+(0,-1) -- ++(4,-1);

\draw [red, thick] (O3)+(0.5,-0.25) .. controls ++(1.5,0) and ++(-1,0) .. ++(2,-1);
\draw [red, dashed] (O33)+(-0.5,0.25) .. controls ++(-1.5,0) and ++(0.75,0) .. ++(-2,-1);
\draw [blue, dashed] (O3)+(-0.5,0.25) .. controls ++(2,0) and ++(-1,0) .. ++(2,1);
\draw [blue, thick] (O33)+(0.5,-0.25) .. controls ++(-2,0) and ++(1,0) .. ++(-2,1);

\end{tikzpicture}
\caption{Smooth line, cusps, and corresponding links}\label{fig:twist}
\end{figure}

\begin{remark}
\label{rem:move}
The moves of divides with cusps in Figure \ref{fig:moves} 
do not change the isotopy type of the corresponding links. 
The proof is similar to that of \cite[Lemma 1.3]{cou-per}. 
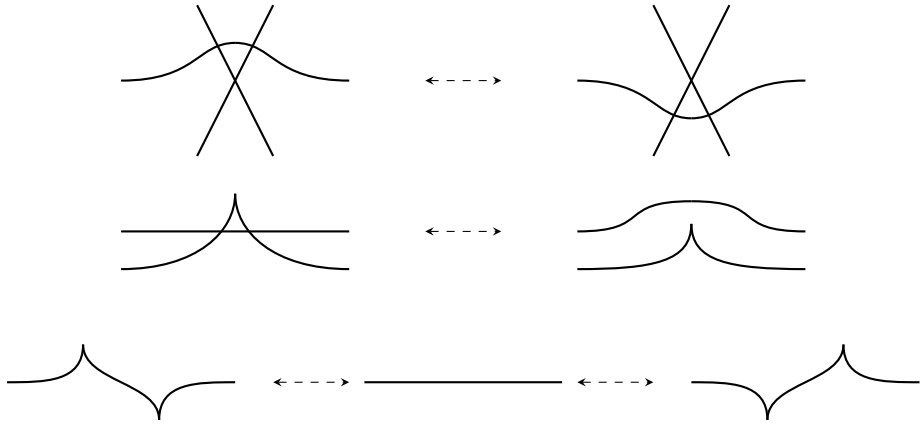
\begin{figure}[htbp]
\centering
\begin{tikzpicture}

\coordinate (R) at (0,4); 
\coordinate (C1) at (0,2);
\coordinate (C2) at (0,0);

\draw[<->, >=stealth, dashed] (R)++(-0.5, 0) -- ++(1,0);
\draw[thick] (R)++(-3.5, -1) -- ++(1, 2);
\draw[thick] (R)++(-3.5, 1) -- ++(1, -2);
\draw[thick] (R)++(-4.5, 0) .. controls ++(1,0) and ++(-0.5,0) .. ++(1.5, 0.5);
\draw[thick] (R)++(-1.5, 0) .. controls ++(-1,0) and ++(0.5,0) .. ++(-1.5, 0.5);

\draw[thick] (R)++(3.5, -1) -- ++(-1, 2);
\draw[thick] (R)++(3.5, 1) -- ++(-1, -2);
\draw[thick] (R)++(4.5, 0) .. controls ++(-1,0) and ++(0.5,0) .. ++(-1.5, -0.5);
\draw[thick] (R)++(1.5, 0) .. controls ++(1,0) and ++(-0.5,0) .. ++(1.5, -0.5);

\draw[<->, >=stealth, dashed] (C1)++(-0.5, 0) -- ++(1,0);
\draw[thick] (C1)++(-4.5,0) --++(3,0); 
\draw[thick] (C1)++(-4.5,-0.5) .. controls ++(1,0) and ++(0,-0.5) .. ++(1.5, 1); 
\draw[thick] (C1)++(-1.5,-0.5) .. controls ++(-1,0) and ++(0,-0.5) .. ++(-1.5, 1); 
\draw[thick] (C1)++(+4.5,0) .. controls ++(-1,0) and ++(1, 0) .. ++(-1.5, 0.4); 
\draw[thick] (C1)++(+1.5,0) .. controls ++(1,0) and ++(-1, 0) .. ++(1.5, 0.4); 
\draw[thick] (C1)++(4.5,-0.5) .. controls ++(-1,0) and ++(0,-0.5) .. ++(-1.5, 0.6); 
\draw[thick] (C1)++(1.5,-0.5) .. controls ++(1,0) and ++(0,-0.5) .. ++(1.5, 0.6); 

\draw[thick] (C2)+(-1.3,0) -- +(1.3,0); 
\draw[<->, >=stealth, dashed] (C2)++(-2.5, 0) -- ++(1,0);
\draw[<->, >=stealth, dashed] (C2)++(1.5, 0) -- ++(1,0);
\draw[thick] (C2)++(-6,0) .. controls ++(0.5,0) and ++(0,-0.5) .. ++(1,0.5);
\draw[thick] (C2)++(-5,0.5) .. controls ++(0,-0.5) and ++(0,0.5) .. ++(1,-1);
\draw[thick] (C2)++(-3,0) .. controls ++(-0.5,0) and ++(0,0.5) .. ++(-1,-0.5);
\draw[thick] (C2)++(6,0) .. controls ++(-0.5,0) and ++(0,-0.5) .. ++(-1,0.5);
\draw[thick] (C2)++(5,0.5) .. controls ++(0,-0.5) and ++(0,0.5) .. ++(-1,-1);
\draw[thick] (C2)++(3,0) .. controls ++(0.5,0) and ++(0,0.5) .. ++(1,-0.5);

\end{tikzpicture}
\caption{Moves of divides with cusps.}\label{fig:moves}
\end{figure}
\end{remark}

\begin{example}
Let $C_1$, $C_2$, and $C_3$ be, respectively, 
a smooth circle, a circle with an outward cusp, 
and a circle with an inward cusp. 
Then $L(C_1)$, $L(C_2)$, and $L(C_3)$ are, respectively, 
the Hopf link, the trivial knot, and the trefoil. 
(Figure \ref{fig:link}.) 

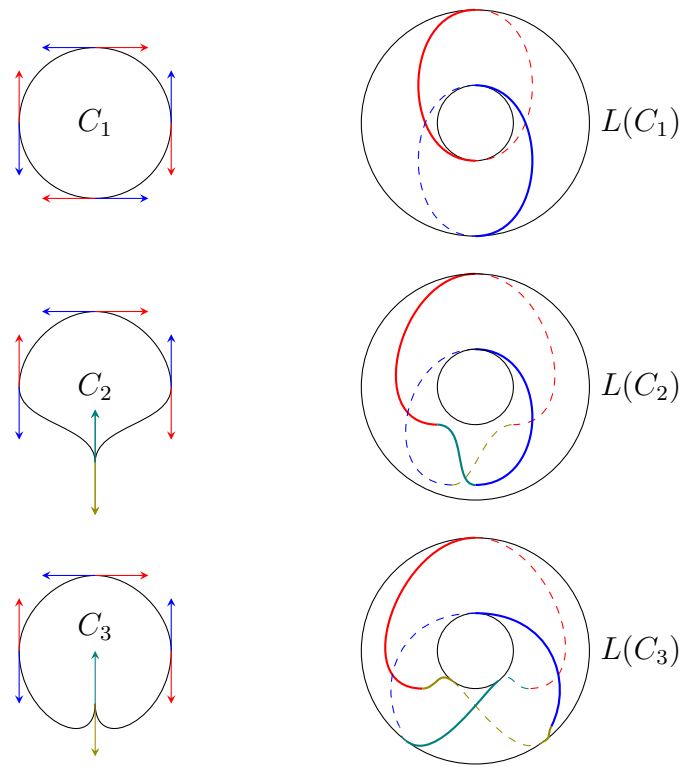
\begin{figure}[htbp]
\centering
\begin{tikzpicture}

\coordinate (O1) at (0,7);
\coordinate (P1) at (5,7);
\coordinate (O2) at (0,3.5);
\coordinate (P2) at (5,3.5);
\coordinate (O3) at (0,0);
\coordinate (P3) at (5,0);

\draw (O1) node {$C_1$}; 
\draw (O1) circle (1);
\draw [->,>=stealth,red] (O1)+(1,0)--++(1,-0.7);
\draw [->,>=stealth,blue] (O1)+(1,0)--++(1,0.7);

\draw [->,>=stealth,red] (O1)+(0,1)--++(0.7,1);
\draw [->,>=stealth,blue] (O1)+(0,1)--++(-0.7,1);

\draw [->,>=stealth,red] (O1)+(-1,0)--++(-1,0.7);
\draw [->,>=stealth,blue] (O1)+(-1,0)--++(-1,-0.7);

\draw [->,>=stealth,blue] (O1)+(0,-1)--++(0.7,-1);
\draw [->,>=stealth,red] (O1)+(0,-1)--++(-0.7,-1);

\draw(P1) circle (0.5);
\draw(P1) circle (1.5);
\draw(P1)+(1.5, 0) node [right] {$L(C_1)$};

\draw[red,thick] (P1)+(0,1.5) .. controls ++(-1,0) and ++(-1,0) .. ++(0,-0.5);
\draw[red, dashed] (P1)+(0,1.5) .. controls ++(1,0) and ++(1,0) .. ++(0,-0.5);

\draw[blue, dashed] (P1)+(0,0.5) .. controls ++(-1,0) and ++(-1,0) .. ++(0,-1.5);
\draw[blue,thick] (P1)+(0,0.5) .. controls ++(1,0) and ++(1,0) .. ++(0,-1.5);

\draw (O2) node {$C_2$}; 
\draw (O2)+(-1,0) .. controls ++(0,0.5) and ++(-0.5,0) .. ++(0,1);
\draw (O2)+(1,0) .. controls ++(0,0.5) and ++(0.5,0) .. ++(0,1);
\draw (O2)+(-1,0) .. controls ++(0,-0.5) and ++(0,0.5) .. ++(0,-1);
\draw (O2)+(1,0) .. controls ++(0,-0.5) and ++(0,0.5) .. ++(0,-1);

\draw [->,>=stealth,red] (O2)+(1,0)--++(1,-0.7);
\draw [->,>=stealth,blue] (O2)+(1,0)--++(1,0.7);

\draw [->,>=stealth,red] (O2)+(0,1)--++(0.7,1);
\draw [->,>=stealth,blue] (O2)+(0,1)--++(-0.7,1);

\draw [->,>=stealth,red] (O2)+(-1,0)--++(-1,0.7);
\draw [->,>=stealth,blue] (O2)+(-1,0)--++(-1,-0.7);

\draw [->,>=stealth,teal] (O2)+(0,-1)--++(0,-0.3);
\draw [->,>=stealth,olive] (O2)+(0,-1)--++(0,-1.7);

\draw[red,thick] (P2)+(0,1.5) .. controls ++(-1,0) and ++(-1,0) .. ++(-0.5,-0.5);
\draw[red, dashed] (P2)+(0,1.5) .. controls ++(1,0) and ++(1,0) .. ++(0.5,-0.5);

\draw[blue, dashed] (P2)+(0,0.5) .. controls ++(-1,0) and ++(-1,0) .. ++(-0.3,-1.3);
\draw[blue,thick] (P2)+(0,0.5) .. controls ++(1,0) and ++(1,0) .. ++(0,-1.3);

\draw[teal, thick] (P2)+(-0.5,-0.5) .. controls ++(0.4,0) and ++(-0.3,0) .. ++(0,-1.3);
\draw[olive, dashed] (P2)+(-0.3,-1.3) .. controls ++(0.3,0) and ++(-0.4,0) .. ++(0.5,-0.5);

\draw [->,>=stealth,teal] (O2)+(0,-1)--++(0,-0.3);
\draw [->,>=stealth,olive] (O2)+(0,-1)--++(0,-1.7);

\draw(P2) circle (0.5);
\draw(P2) circle (1.5);
\draw(P2)+(1.5, 0) node [right] {$L(C_2)$};

\draw (O3) node [above] {$C_3$}; 
\draw (O3)+(-1,0) .. controls ++(0,0.5) and ++(-0.5,0) .. ++(0,1);
\draw (O3)+(1,0) .. controls ++(0,0.5) and ++(0.5,0) .. ++(0,1);
\draw (O3)+(-1,0) .. controls ++(0,-0.9) and ++(0,-0.7) .. ++(0,-0.7);
\draw (O3)+(1,0) .. controls ++(0,-0.9) and ++(0,-0.7) .. ++(0,-0.7);

\draw [->,>=stealth,red] (O3)+(1,0)--++(1,-0.7);
\draw [->,>=stealth,blue] (O3)+(1,0)--++(1,0.7);

\draw [->,>=stealth,red] (O3)+(0,1)--++(0.7,1);
\draw [->,>=stealth,blue] (O3)+(0,1)--++(-0.7,1);

\draw [->,>=stealth,red] (O3)+(-1,0)--++(-1,0.7);
\draw [->,>=stealth,blue] (O3)+(-1,0)--++(-1,-0.7);

\draw [->,>=stealth,teal] (O3)+(0,-0.7)--++(0,0);
\draw [->,>=stealth,olive] (O3)+(0,-0.7)--++(0,-1.4);

\draw(P3) circle (0.5);
\draw(P3) circle (1.5);
\draw(P3)+(1.5, 0) node [right] {$L(C_3)$};

\draw[red, thick] (P3)+(0,1.5) .. controls ++(-1,0) and ++(-1,0) .. ++(-0.7,-0.5);
\draw[red, dashed] (P3)+(0,1.5) .. controls ++(1,0) and ++(1,0) .. ++(0.7,-0.5);

\draw[olive, thick] (P3)+(-0.7,-0.5) .. controls ++(0.2,0) and ++(-0.2,0.15) .. ++(-0.3,-0.4);
\draw[olive, thick] (P3)+(0.9,-1.2) .. controls ++(0.1,0.05) and ++(-0.06,-0.12) .. ++(1,-1);
\draw[olive, dashed] (P3)+(-0.3,-0.4) .. controls ++(0.2,-0.15) and ++(-0.4,-0.3) .. ++(0.9,-1.2);

\draw[blue, dashed] (P3)+(0,0.5) .. controls ++(-1,0) and ++(-0.3,0.6) .. ++(-1,-1);
\draw[blue, thick] (P3)+(0,0.5) .. controls ++(1,0) and ++(0.3,0.6) .. ++(1,-1);

\draw[teal, dashed] (P3)+(0.7,-0.5) .. controls ++(-0.2,0) and ++(0.2,0.15) .. ++(0.3,-0.4);
\draw[teal, dashed] (P3)+(-0.9,-1.2) .. controls ++(-0.1,0.05) and ++(0.06,-0.12) .. ++(-1,-1);
\draw[teal, thick] (P3)+(0.3,-0.4) .. controls ++(-0.2,-0.15) and ++(0.4,-0.3) .. ++(-0.9,-1.2);

\end{tikzpicture}
\caption{Examples of $L(C)$ (drawn on the torus)}\label{fig:link}
\end{figure}
\end{example}


\section{Handle decomposition of $M(\A)$}

\label{sec:handle}






\subsection{Setting}
\label{sec:setting}

A real line arrangement $\A =\{H_1, \dots, H_n\}$ is a 
finite set of affine lines in the affine plane 
$\R^2$. Each line is defined by some affine 
linear form 
\begin{equation}
\label{eq:defining}
\alpha_H(x_1, x_2)=ax_1+bx_2+c=0, 
\end{equation}
with $a, b, c\in\R$ and $(a, b)\neq (0,0)$. 
A connected component of $\R^2\smallsetminus\bigcup_{H\in\A}H$ 
is called a chamber. The set of all chambers is denoted by 
$\ch(\A)$. The affine linear equation (\ref{eq:defining}) 
defines a complex line 
$\{(z_1, z_2)\in\C^2\mid az_1+bz_2+c=0\}$ in $\C^2$. 
We denote the set of complexified lines by 
$\A_\C=\{H_\C=H\otimes\C\mid H\in\A\}$. 
The object of our interest is the complexified 
complement $M(\A)=\C^2\smallsetminus\bigcup_{H\in\A}H_\C$. 

For a point $p\in\R^2$, let $\A_p:=\{H\in\A\mid H\ni p\}$. 
By the identification 
$T\R^2\simeq\C^2$, 
$(\bm{x}, \bm{y}\in T_{\bm{x}}\R^2)\longmapsto \bm{x}+\sqrt{-1}\cdot\bm{y}$, 
the complexified complement $M(\A)$ is the set of all tangent vectors in $\R^2$ 
which are tangent to none of $H\in\A$, namely, 
\begin{equation}
\label{eq:compl}
M(\A)=\{(\bm{x}, \bm{y})\mid \bm{x}\in\R^2, \bm{y}\in T_{\bm{x}}\R^2, \mbox{ if 
$H\in\A_{\bm{x}}$, then $\bm{y}\notin T_{\bm{x}}H$}
\}. 
\end{equation}
Based on this description, we have the following. 

\begin{proposition}
\label{prop:easy}
\begin{itemize}
\item[(1)] 
Let $(\bm{x}, \bm{y})\in M(\A)$. 
Then for any $t\in\R$, $(\bm{x}+t\bm{y}, \bm{y})\in M(\A)$. 
\item[(2)] Suppose $\bm{y}$ is not parallel to any of $H\in\A$. 
Then for any $\bm{x}\in\R^2$, $(\bm{x}, \bm{y})\in M(\A)$. 
\end{itemize}
\end{proposition}
\begin{proof}
(1) Assume $(\bm{x}+t\bm{y}, \bm{y})\in H_{\C}$ for some $H\in\A$. 
Then $\bm{x}+t\bm{y}\in H$ and $\bm{y}\in T_{\bm{x}+t\bm{y}}H$ hold. 
Then $\bm{y}$ is parallel to $H$, and $\bm{x}\in H$. 
Then $(\bm{x}, \bm{y})\in H$ which contradicts $(\bm{x}, \bm{y})\in M(\A)$. 

(2) is obvious from the description of (\ref{eq:compl}). 
\end{proof}

Now let us fix a generic line $F$ such that 
$F$ does not separate intersections of $\A$. 
We can choose coordinates $x_1, x_2$ so that 
$F$ is given by $x_2=0$ and all intersections of $\A$ 
are contained in the upper half-plane 
$\{(x_1, x_2)\in\R^2\mid x_2>0\}$. 
Let 
\[
\ch_F(\A):=\{C\in\ch(\A)\mid C\cap F=\emptyset\}. 
\]
We set $H_i\cap F$ has coordinates $(a_i, 0)$. 
By changing the numbering of lines and signs of 
the defining equation $\alpha_i$ of $H_i\in\A$ 
we may assume that 
\begin{itemize}
\item $a_1<a_2<\dots<a_n$, and 
\item the half-plane $H_i^-=\{\alpha_i<0\}$ contains the negative 
direction (leftward) of $F$ (equivalently, $H_i^+=\{\alpha_i>0\}$ contains 
the positive direction).  
\end{itemize}
Let us recall the construction in \cite{yos-lef} briefly. 
There exists a (after suitable compactification near the boundary) 
Morse function $\varphi:M(\A)\longrightarrow\R_{\geq 0}$ 
such that 
\begin{itemize}
\item $\varphi^{-1}(0)=M(\A)\cap F_{\C}$, 
\item for each $C\in\ch_F(\A)$ there is a unique critical point 
$p_C\in C$,  
\item the chamber $C$ is the stable manifold of $p_C$, 
\item there are no other critical points. 
\end{itemize}
Based on these facts, it was proved that 
$M(\A)\smallsetminus\bigsqcup_{C\in\ch_F(\A)}C$ is 
homeomorphic to an open neighborhood of $F_\C\cap M(\A)$. 
The homeomorphism/retraction was originally constructed by 
using the gradient flow (\cite[Corollary 5.1.5]{yos-lef}). 
Below we construct a piecewise linear retraction. 


After appropriate changes of coordinates, we may further assume the following. 
(See Figure \ref{fig:setting}.)
\begin{itemize}
\item 
There exists a positive real number $R_0>0$ such that all the intersection 
points of $\A$ are lying in $\{(x_1, x_2)\mid 1<x_2<R_0\}$. 
\item 
Denote by $\theta_i$ the angle between the $x_1$-axis (from the positive side) and 
$H_i$. Then 
$\frac{\pi}{4}\leq \theta_1\leq\theta_2\leq\dots\leq\theta_n\leq\frac{3\pi}{4}$. 
\item 
Take a sufficiently large real number $R>0$ ($R\gg R_0$) so that 
all intersections $H_i\cap \{x_2=R_0\}$, $i=1, \dots, n$, are contained in 
$\{-R+R_0<x_1<R-R_0\}$. 
\end{itemize}

\begin{figure}[htbp]
\centering
\begin{tikzpicture}

\coordinate (a1) at (-1,0);
\coordinate (a2) at (0,0);
\coordinate (a3) at (1,0);

\draw[dashed] (-7,0) -- (7,0);
\draw (6,0.3) node[right] {\small $F=\{x_2=0\}$};

\draw[loosely dashed] (-5,-1) rectangle (5,4.5);
\draw[red] (-5,-1) node [below] {\small $\Rect(R, 1)$} rectangle (5,1);

\fill[black] (a1) node [below right] {\small $a_1$} circle(0.06);
\draw [thick] (a1) -- ++(75:5.2cm) node [right] {$H_1$};
\draw [thick] (a1) -- ++(255:1.3cm);
\draw (a1)+(0.3,0)  arc (0:75:0.3) node [above right] {\footnotesize $\theta_1$};

\fill[black] (a2) node [below right] {\small $a_2$} circle(0.06);
\draw [thick] (a2) -- ++(90:5cm) node [left] {$H_2$};
\draw [thick] (a2) -- ++(270:1.2cm);
\draw (a2)+(0.3,0)  arc (0:90:0.3) node [above right] {\footnotesize $\theta_2$};

\fill[black] (a3) node [below right] {\small $a_3$} circle(0.06);
\draw [thick] (a3) -- ++(115:5.4cm) node [left] {$H_3$};
\draw [thick] (a3) -- ++(295:1.3cm);
\draw (a3)+(0.3,0)  arc (0:115:0.3) node [above right] {\footnotesize $\theta_3$};

\draw (5,3) node [right] {\small $\{x_1=R\}$};
\draw (-5,3) node [left] {\small $\{x_1=-R\}$};
\draw (3,4.5) node [above] {\small $\{x_2=R_0\}$};
\draw (3,1) node [above] {\small $\{x_2=1\}$};
\draw (3,-1) node [below] {\small $\{x_2=-1\}$};

\draw[dashed, blue] (-5.4, -0.5) node [left] {\small $x_2=x_1+R$} -- ++(4.4, 5.5);
\draw[dashed, blue] (5.4, -0.5) node [right] {\small $x_2=-x_1+R$} -- ++(-4.4, 5.5);

\end{tikzpicture}
\caption{Setting}\label{fig:setting}
\end{figure}

Let $\widetilde{H}_{\C}$ be a small tubular neighborhood of $H_{\C}$. 
Define $M_1$ as 
\begin{equation}
M_1=
\left(\C^2\smallsetminus\bigcup_{H\in\A}\widetilde{H}_{\C}\right)\cap 
\{(\bm{x}, \bm{y})\mid\bm{x}\in[-R, R]\times [-1, R_0], ||\bm{y}||_{\infty}\leq R\}, 
\end{equation}
which is a compact $4$-dimensional manifold with boundary whose 
interior is homeomorphic to $M(\A)$ (see \cite{dur}). 

Next, we remove tubular neighborhoods of each chamber $C$. 
Let $\widetilde{C}$ be the tubular neighborhood of $C$ in $M(\A)$. 
Define $M_2$ as 
\begin{equation}
M_2=
\left(
M_1\smallsetminus \bigcup_{C\in\ch(\A)}\widetilde{C}
\right)
\cup
\{(\bm{x}, \bm{y})\mid\bm{x}\in\Rect(R, 1), ||\bm{y}||_{\infty}\leq R\}
\smallsetminus\bigcup_{H\in\A}\widetilde{H}_{\C}. 
\end{equation}
Note that if $C\cap F\neq\emptyset$, then the addition of 
$\widetilde{C}$ does not change the topology. 
On the other hand, $\widetilde{C}$ for 
$C\in\ch_F(\A)$ corresponds to a $2$-handle. 
Hence $M(\A)$ is obtained from $M_2$ by attaching $2$-handles 
$\widetilde{C}$ for each $C\in\ch_F(\A)$. 

Let $\rho:[1, R_0]\stackrel{\simeq}{\longrightarrow}[0, 1]$ be the function 
defined by 
$\rho(h)=\frac{h-1}{R_0-1}$. Next, we define 
\begin{equation}
M_3=
\left\{(\bm{x}, \bm{y})\in M_2\left|\ 
\begin{split}
&\mbox{If $x_2\geq 1$, then}\\
&x_2\leq x_1+R, x_2\leq -x_1+R, \\
&\mbox{and }\rho(x_2)\leq ||\bm{y}||_{\infty}\leq 1
\end{split}
\right.\right\}. 
\end{equation}
Note that the inequalities $x_2\leq x_1+R, x_2\leq -x_1+R$ mean that 
the real part of the point is below the blue dashed lines in Figure \ref{fig:setting}. 
Since $M_3$ is obtained from $M_2$ just by deforming near the boundary, 
$M_2$ and $M_3$ are homeomorphic. 
Let us define $M_4$ as 
\begin{equation}
M_4:=
\{(\bm{x}, \bm{y})\mid \bm{x}\in \Rect(R, 1), \bm{y}\in T_{\bm{x}}\R^2, 
||\bm{y}||_\infty\leq\delta(\bm{x})\}\smallsetminus\bigcup_{H\in\A}\widetilde{H}_{\C}. 
\end{equation}
(c. f. Figure \ref{fig:M1234}.) 

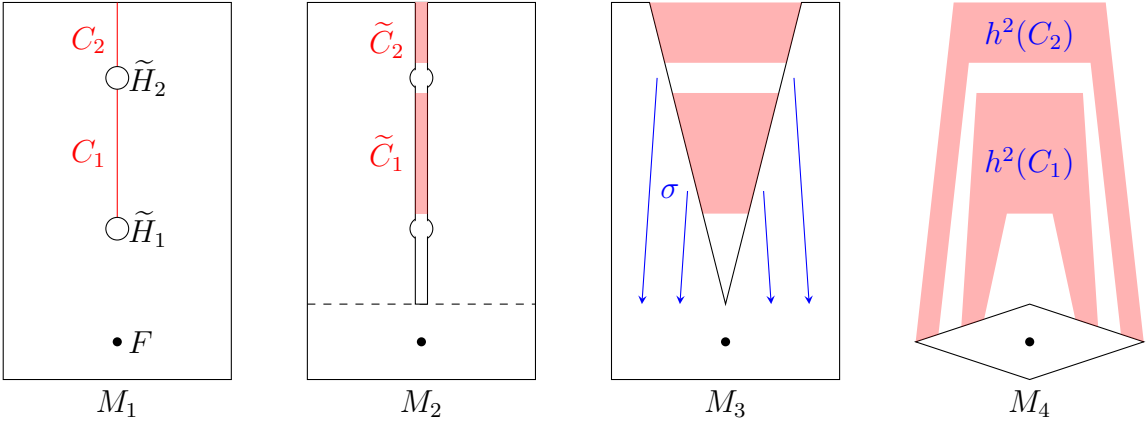
\begin{figure}[htbp]
\centering
\begin{tikzpicture}

\coordinate (M1) at (0,1);
\coordinate (M2) at (4,1);
\coordinate (M3) at (8,1);
\coordinate (M4) at (12,1);


\draw (M1) rectangle ++(3,5);
\fill (M1)+(1.5, 0.5) circle (0.06cm) node [right] {$F$};
\draw[red] (M1)+(1.5, 2)-- node [left] {$C_1$} ++(1.5, 4);
\draw[red] (M1)+(1.5, 4)-- node [left] {$C_2$} ++(1.5, 5);
\filldraw[fill=white, draw=black] (M1)+(1.5, 2) circle (0.15cm) node [right] {$\widetilde{H}_1$};
\filldraw[fill=white, draw=black] (M1)+(1.5, 4) circle (0.15cm) node [right] {$\widetilde{H}_2$};
\draw (M1)+(1.5, 0) node [below] {$M_1$};

\draw (M2) rectangle ++(3,5);
\draw[dashed] (M2)+(0,1)--++(3,1);
\fill (M2)+(1.5, 0.5) circle (0.06cm);
\draw (M2)+(1.5, 2) circle (0.15cm);
\draw (M2)+(1.5, 4) circle (0.15cm);
\filldraw[fill=white, draw=black] (M2)+(1.42, 1) rectangle ++(1.58, 5);
\draw[white] (M2)+(1.42, 5)--++(1.58, 5);
\fill[white] (M2)+(1.5, 2) circle (0.13cm);
\fill[white] (M2)+(1.5, 4) circle (0.13cm);
\fill[red, opacity=.3] (M2)+(1.42, 2.2) rectangle ++(1.58, 3.8);
\draw[red] (M2)+(1.42, 3) node [left] {$\widetilde{C}_1$};
\fill[red, opacity=.3] (M2)+(1.42, 4.2) rectangle ++(1.58, 5);
\draw[red] (M2)+(1.42, 4.5) node [left] {$\widetilde{C}_2$};
\draw (M2)+(1.5, 0) node [below] {$M_2$};

\fill (M3)+(1.5, 0.5) circle (0.06cm);
\draw (M3) -- ++(3,0) --++(0,5) -- ++(-0.5, 0) -- ++(-1, -4) -- ++(-1, 4) -- ++(-0.5, 0) -- cycle;

\fill[red, opacity=.3] (M3)+(0.5, 5) -- ++ (0.7, 4.2) -- ++ (1.6, 0) -- ++ (0.2, 0.8) -- cycle;

\fill[red, opacity=.3] (M3)+(0.8, 3.8) -- ++ (1.2, 2.2) -- ++ (0.6, 0) -- ++ (0.4, 1.6) -- cycle;

\draw [->,>=stealth, blue] (M3)+(0.6,4)-- node [right] {$\sigma$} ++(0.4,1);
\draw [->,>=stealth, blue] (M3)+(1,2.5)--  ++(0.9,1);
\draw [->,>=stealth, blue] (M3)+(2.4,4)--  ++(2.6,1);
\draw [->,>=stealth, blue] (M3)+(2,2.5)--  ++(2.1,1);

\draw (M3)+(1.5, 0) node [below] {$M_3$};

\fill (M4)+(1.5, 0.5) circle (0.06cm);
\draw (M4)+(0,0.5) -- ++(1.5, 0) -- ++ (1.5, 0.5) -- ++ (-1.5, 0.5) -- cycle; 


\fill[red, opacity=.3] (M4)+(0, 0.5) -- ++ (0.5, 5) -- ++ (2, 0) -- ++ (0.5, -4.5) -- ++(-0.3, 0.1) -- ++(-0.4, 3.6) -- ++(-1.6, 0) -- ++(-0.4, -3.6) -- cycle;

\fill[red, opacity=.3] (M4)+(0.6, 0.7) -- ++ (0.8, 3.8) -- ++ (1.4, 0) -- ++ (0.2, -3.1) -- ++ (-0.3, 0.1) -- ++ (-0.3, 1.4) -- ++(-0.6, 0)-- ++(-0.3, -1.4) -- cycle;


\draw (M4)+(1.5, 0) node [below] {$M_4$};
\draw[blue] (M4)+(1.5, 4.2) node [above] {$h^2(C_2)$};
\draw[blue] (M4)+(1.5, 2.5) node [above] {$h^2(C_1)$};

\end{tikzpicture}
\caption{$M_1, M_2 , M_3$, and $M_4$ (and attached $2$-handles).}\label{fig:M1234}

\end{figure}

Now we will construct an explicit contraction from $M_3$ to $M_4$. 

\subsection{Explicit contraction}

As we noted, $M_2, M_3$, and $M_4$ are homeomorphic. 
Define the map  $\sigma:M_3\longrightarrow M_4$ by 
\begin{equation}
\sigma(\bm{x}, \bm{y}):=
\begin{cases}
\left(\left(x_1-\frac{y_1}{y_2}(x_2-1+||\bm{y}||_{\infty}), 1- ||\bm{y}||_{\infty}\right), \bm{y}\right),  &\mbox{ if $x_2\geq 1-||\bm{y}||_{\infty}$ and $|y_1|\leq |y_2|$}, \\
\left(\left(x_1-\frac{y_2}{y_1}(x_2-1+||\bm{y}||_{\infty}), 1- ||\bm{y}||_{\infty}\right), \bm{y}\right),  &\mbox{ if $x_2\geq 1-||\bm{y}||_{\infty}$ and $|y_1|\geq |y_2|$}, \\
(\bm{x}, \bm{y})& \mbox{ if $x_2\leq 1-||\bm{y}||_{\infty}$,} 
\end{cases}
\end{equation}
where $\bm{x}=(x_1, x_2)$ and $\bm{y}=(y_1, y_2)$ (Figure \ref{fig:sigma}). 
Note that 
the first case is for the edges $R_1$ and $R_3$, 
the second case is for the edges $R_2$ and $R_4$, and 
the third case is for $(\bm{x}, \bm{y})\in M_4$. 
\begin{figure}[htbp]
\centering
\begin{tikzpicture}

 \fill [black] (-8.5,2.5) circle (0.06);
 \coordinate [label = above right:$P$] (Q) at (-8.5,2.5);
 \draw [thick](-9,2) -- (-8,2);
 \coordinate [label = below : $R_1$] (R) at (-8.5,1.75);
 
 \draw[->,>=stealth] (Q) -- (-9,2);
 \draw[->,>=stealth] (Q) -- (-26/3,2);
 \draw[->,>=stealth] (Q) -- (-25/3,2);
 \draw[->,>=stealth] (Q) -- (-8,2);

 \coordinate[label = above:$\sigma$] (S) at (-6,2.6);
 \draw [->,>=stealth] (-7,2.5) -- (-5,2.5);

 \draw [domain=-1.5:3] plot(\x , \x+1); 
 \draw [domain=1.5:-3] plot(\x , -\x+1); 
 \draw [thick](-4.5,0) -- (4.5,0);
 \fill [black] (0, 3.5) circle(0.06);
 \coordinate[label=above right:$P$] (P) at (0,3.5);
 
 \coordinate (A) at (-1/2,3);
 \coordinate (B) at (-1/6,3);
 \coordinate (C) at (1/6,3);
 \coordinate (D) at (1/2,3);
 
 \draw [->,>=stealth] (P)--(A); 
 \draw [->,>=stealth] (P)--(B);
 \draw [->,>=stealth] (P)--(C);
 \draw [->,>=stealth] (P)--(D);
 \draw [thick](A) -- (D);
 
 \coordinate (E) at (-3.5,0);
 \coordinate (F) at (-7/6,0);
 \coordinate (G) at (7/6,0);
 \coordinate (H) at (3.5,0);

 \draw [dashed] (A) -- (E);
 \draw [dashed] (B) -- (F);
 \draw [dashed] (C) -- (G);
 \draw [dashed] (D) -- (H);
 
 \fill [black] (E) circle(0.06);
 \fill [black] (F) circle(0.06);
 \fill [black] (G) circle(0.06);
 \fill [black] (H) circle(0.06);
 \draw [->,>=stealth] (E)--(-4,-1/2);
 \draw [->,>=stealth] (F)--(-4/3,-1/2);
 \draw [->,>=stealth] (G)--(4/3,-1/2);
 \draw [->,>=stealth] (H)--(4,-1/2);
\end{tikzpicture}


\begin{tikzpicture}
 \fill [black] (-8.5,2.5) circle (0.06);
 \coordinate [label = above left:$P$] (Q) at (-8.5,2.5);
 \draw [thick](-8,2) -- (-8,3);
 \coordinate [label = below : $R_2$] (R) at (-8.5,1.75);
 
 \draw[->,>=stealth] (Q) -- (-8,3);
 \draw[->,>=stealth] (Q) -- (-8,8/3);
 \draw[->,>=stealth] (Q) -- (-8,7/3);
 \draw[->,>=stealth] (Q) -- (-8,2);

 \coordinate[label = above:$\sigma$] (S) at (-6,2.6);
 \draw [->,>=stealth] (-7,2.5) -- (-5,2.5);

 \draw [domain=-1.5:3] plot(\x , \x+1); 
 \draw [domain=1.5:-3] plot(\x , -\x+1); 
 \draw [thick](-4.5,0) -- (4.5,0);
 \fill [black] (0, 3.5) circle(0.06);
 \coordinate[label=above left:$P$] (P) at (0,3.5);
 
 \coordinate (A) at (1/2,3);
 \coordinate [label = right:$A$](B) at (1/2,10/3);
 \coordinate [label = right:$B$](C) at (1/2,11/3);
 \coordinate (D) at (1/2,4);
 
 \draw [->,>=stealth] (P)--(A); 
 \draw [->,>=stealth] (P)--(B);
 \draw [->,>=stealth] (P)--(C);
 \draw [->,>=stealth] (P)--(D);
 \draw [thick](A) -- (D);
 
 \coordinate (E) at (-3.5,0);
 \coordinate [label =above left:$\sigma(B)$](F) at (-7/6,0);
 \coordinate [label =above right:$\sigma(A)$](G) at (7/6,0);
 \coordinate (H) at (3.5,0);

 \draw [dashed] (P) -- (E);
 \draw [dashed] (P) -- (F);
 \draw [dashed] (P) -- (G);
 \draw [dashed] (P) -- (H);
 
 \fill [black] (E) circle(0.06);
 \fill [black] (F) circle(0.06);
 \fill [black] (G) circle(0.06);
 \fill [black] (H) circle(0.06);
 \draw [->,>=stealth] (E)--(-3,1/2);
 \draw [->,>=stealth] (F)--(-2/3,1/6);
 \draw [->,>=stealth] (G)--(5/3,-1/6);
 \draw [->,>=stealth] (H)--(4,-1/2);
 
 \coordinate (I) at (-1/2,2);
 \coordinate (J) at (1/2,2);
 \fill [black] (I) circle (0.06);
 \fill [black] (J) circle (0.06);
 
 \draw [->,>=stealth] (I)--(0,13/6);
 \draw [->,>=stealth] (J)--(1,11/6);
\end{tikzpicture}


\begin{tikzpicture}
 \fill [black] (-8.5,2.5) circle (0.06);
 \coordinate [label = below right:$P$] (Q) at (-8.5,2.5);
 \draw [thick](-9,3) -- (-8,3);
 \coordinate [label = below : $R_3$] (R) at (-8.5,1.75);
 
 \draw[->,>=stealth] (Q) -- (-9,3);
 \draw[->,>=stealth] (Q) -- (-26/3,3);
 \draw[->,>=stealth] (Q) -- (-25/3,3);
 \draw[->,>=stealth] (Q) -- (-8,3);

 \coordinate[label = above:$\sigma$] (S) at (-6,2.6);
 \draw [->,>=stealth] (-7,2.5) -- (-5,2.5);

 \draw [domain=-1.5:3] plot(\x , \x+1); 
 \draw [domain=1.5:-3] plot(\x , -\x+1); 
 \draw [thick](-4.5,0) -- (4.5,0);
 \fill [black] (0, 3.5) circle(0.06);
 \coordinate[label= right:$P$] (P) at (0,3.5);
 
 \coordinate (A) at (-1/2,4);
 \coordinate (B) at (-1/6,4);
 \coordinate (C) at (1/6,4);
 \coordinate (D) at (1/2,4);
 
 \draw [->,>=stealth] (P)--(A); 
 \draw [->,>=stealth] (P)--(B);
 \draw [->,>=stealth] (P)--(C);
 \draw [->,>=stealth] (P)--(D);
 \draw [thick](A) -- (D);
 
 \coordinate (E) at (-3.5,0);
 \coordinate (F) at (-7/6,0);
 \coordinate (G) at (7/6,0);
 \coordinate (H) at (3.5,0);

 \draw [dashed] (P) -- (E);
 \draw [dashed] (P) -- (F);
 \draw [dashed] (P) -- (G);
 \draw [dashed] (P) -- (H);
 
 \fill [black] (E) circle(0.06);
 \fill [black] (F) circle(0.06);
 \fill [black] (G) circle(0.06);
 \fill [black] (H) circle(0.06);
 \draw [->,>=stealth] (E)--(-3,1/2);
 \draw [->,>=stealth] (F)--(-1,1/2);
 \draw [->,>=stealth] (G)--(1,1/2);
 \draw [->,>=stealth] (H)--(3,1/2);
\end{tikzpicture}


\begin{tikzpicture}
 \fill [black] (-8.5,2.5) circle (0.06);
 \coordinate [label = above right:$P$] (Q) at (-8.5,2.5);
 \draw [thick](-9,2) -- (-9,3);
 \coordinate [label = below : $R_4$] (R) at (-8.5,1.75);
 
 \draw[->,>=stealth] (Q) -- (-9,2);
 \draw[->,>=stealth] (Q) -- (-9,7/3);
 \draw[->,>=stealth] (Q) -- (-9,8/3);
 \draw[->,>=stealth] (Q) -- (-9,3);

 \coordinate[label = above:$\sigma$] (S) at (-6,2.6);
 \draw [->,>=stealth] (-7,2.5) -- (-5,2.5);

 \draw [domain=-1.5:3] plot(\x , \x+1); 
 \draw [domain=1.5:-3] plot(\x , -\x+1); 
 \draw [thick](-4.5,0) -- (4.5,0);
 \fill [black] (0, 3.5) circle(0.06);
 \coordinate[label=above right:$P$] (P) at (0,3.5);
 
 \coordinate (A) at (-1/2,3);
 \coordinate [label = left:$A$](B) at (-1/2,10/3);
 \coordinate [label = left:$B$](C) at (-1/2,11/3);
 \coordinate (D) at (-1/2,4);
 
 \draw [->,>=stealth] (P)--(A); 
 \draw [->,>=stealth] (P)--(B);
 \draw [->,>=stealth] (P)--(C);
 \draw [->,>=stealth] (P)--(D);
 \draw [thick](A) -- (D);
 
 \coordinate (E) at (-3.5,0);
 \coordinate [label =above left:$\sigma(A)$](F) at (-7/6,0);
 \coordinate [label =above right:$\sigma(B)$](G) at (7/6,0);
 \coordinate (H) at (3.5,0);

 \draw [dashed] (P) -- (E);
 \draw [dashed] (P) -- (F);
 \draw [dashed] (P) -- (G);
 \draw [dashed] (P) -- (H);
 
 \fill [black] (E) circle(0.06);
 \fill [black] (F) circle(0.06);
 \fill [black] (G) circle(0.06);
 \fill [black] (H) circle(0.06);
 \draw [->,>=stealth] (E)--(-4,-1/2);
 \draw [->,>=stealth] (F)--(-5/3,-1/6);
 \draw [->,>=stealth] (G)--(2/3,1/6);
 \draw [->,>=stealth] (H)--(3,1/2);
 
 \coordinate (I) at (-1/2,2);
 \coordinate (J) at (1/2,2);
 \fill [black] (I) circle (0.06);
 \fill [black] (J) circle (0.06);
 
 \draw [->,>=stealth] (I)--(-1,11/6);
 \draw [->,>=stealth] (J)--(0,13/6);
\end{tikzpicture}
\caption{$\sigma:M_3 \rightarrow M_4$}\label{fig:sigma}

\end{figure}
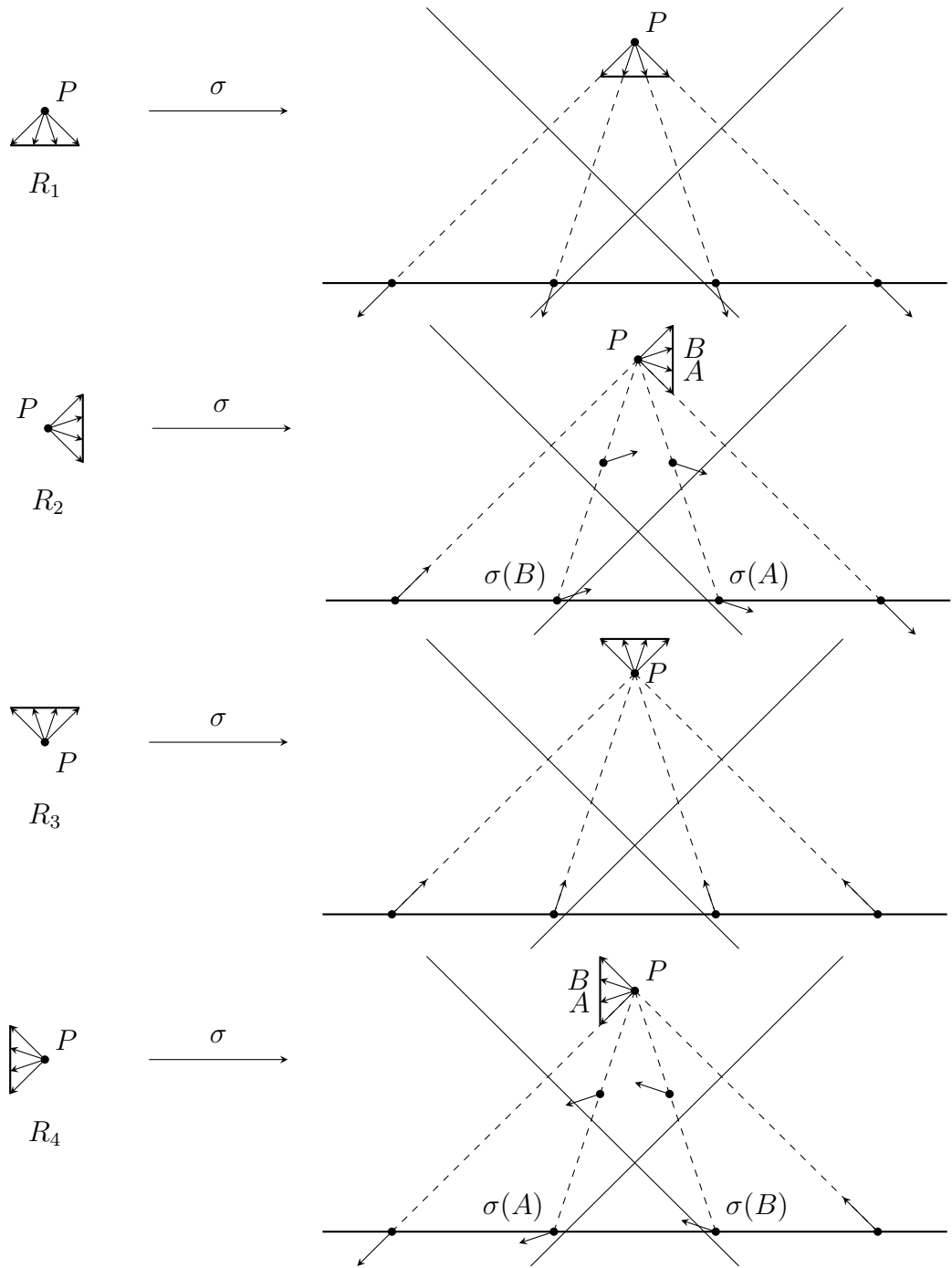

Let $\sigma_t(\bm{x}, \bm{y})=(1-t)\cdot (\bm{x}, \bm{y})+
t\cdot\sigma(\bm{x}, \bm{y})$. It is straightforward that 
if $(\bm{x}, \bm{y})\in M_3$, 
then $\sigma_t(\bm{x}, \bm{y})\in M_3$ for $0\leq t\leq 1$. 
Since $\sigma_0$ is the identity on $M_3$ and 
$\sigma_1=\sigma$, $\sigma_t$ ($0\leq t\leq 1$) 
gives a contraction of $M_3$ to $M_4$. 

\subsection{Attaching maps of $2$-handles}

In this section, we describe the Kirby diagrams  
(the framed links of attaching circles and the boundaries of carved disks) 
associated with the handle decompositions obtained in the previous section. 
Our $0$-handle, namely $D^4$, is the piecewise linear disk that we introduced in 
\S \ref{sec:sphere} 
\begin{equation}
\label{eq:D4}
\disk(R, 1)\approx 
\{(\bm{x}, \bm{y})\mid \bm{x}\in \Rect(R, 1), \bm{y}\in T_{\bm{x}}\R^2, 
||\bm{y}||_\infty\leq\delta(\bm{x})\}. 
\end{equation}
Note that $\partial \disk(R, 1)=\sphere(R, 1)$ and $\disk(R, 1)\cap M_1=M_4$. 

Next, we consider $1$-handlebody. Note that the intersection 
$\disk(R, 1)\cap M_1$ is equal to 
$\disk(R, 1)\smallsetminus\bigsqcup_{i=1}\widetilde{H}_{i, \C}$. 
The intersection 
$\disk(R, 1)\cap H_{i, \C}$ is homeomorphic to a disk whose boundary 
$\partial(\disk(R, 1)\cap H_{i, \C})=\sphere(R, 1)\cap H_{i, \C}$ 
is a circle. This circle is presented as ``dotted circle'' in 
the Kirby diagram (\cite{akb, kir-top, gom-sti}). 
We note that $\sphere(R, 1)\cap H_{i, \C}$ is the knot corresponding to 
the divide $\overline{H}_i:=H_i\cap\Rect(R, 1)$. 
Thus the Kirby diagram of the $1$-handlebody is the link associated with 
the divide consisting of $n$ intervals 
$\{\overline{H_1}, \dots, \overline{H}_n\}$ in $\Rect(R, 1)$ 
(Figure \ref{fig:1handle}). 

\begin{figure}[htbp]
\centering
\begin{tikzpicture}

\draw[thick] (-5,-1)  rectangle (5,1);
\draw (0, -1) node [below] {$\Rect(R, 1)$}; 
\draw[thick] (-4,-1) -- node[right] {$\overline{H}_1$} (-2,1);
\filldraw (-2.3, 0.7) circle (0.1);
\draw[thick] (-1,-1) -- node[right] {$\overline{H}_2$} (0,1);
\filldraw (-0.15, 0.7) circle (0.1);
\draw[thick] (3,-1) -- node[right] {$\overline{H}_3$} (1.5,1);
\filldraw (1.725, 0.7) circle (0.1);

\end{tikzpicture}
\caption{Divide whose link (with dots) represents the $1$-handlebody}\label{fig:1handle}
\end{figure}
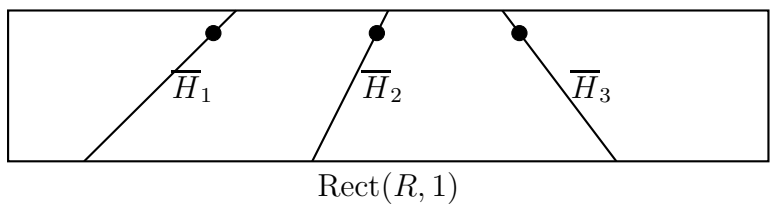

From the description in the previous section, to each chamber $C\in\ch_F(\A)$, 
the union of the neighborhood of $C$ 
\begin{equation}
\{(\bm{x}, \bm{y})\mid \bm{x}\in C, ||\bm{y}||_{\infty}\leq\rho(x_2)\}
\end{equation}
and the trajectories of the contraction 
\begin{equation}
\{\sigma_t(\bm{x}, \bm{y})
\mid \bm{x}\in C, ||\bm{y}||_{\infty}=\rho(x_2), 0\leq t\leq 1\}
\end{equation}
gives an explicit $2$-handle $h^2(C)$ attached to $M_4$ (Figure \ref{fig:M1234}). 

Let $P=(x_1, x_2)\in C$. 
The real part of $\sigma(x_1, x_2, -\rho(x_2), -\rho(x_2))$ 
is $(x_1-x_2+1-\rho(x_2), 1-\rho(x_2))$ which we denote by $\bm{a}_1(P)$. 
Similarly, the real part of $\sigma(x_1, x_2, \rho(x_2), -\rho(x_2))$ 
is $(x_1+x_2-1+\rho(x_2), 1-\rho(x_2))$ which we denote by $\bm{a}_2(P)$. 

\begin{theorem}
\label{thm:attach}
The attaching circle of the $2$-handle $h^2(C)$ for $C\in\ch_F(\A)$ is 
equal to $\FS(\bm{a}_1, \bm{a}_2)$ with zero framing. 
(Figure \ref{fig:attc}) 
\end{theorem}
\begin{proof}
The first assertion is clear from the definition of $\sigma$. 
The framing is computed as the linking number of 
$\FS(\bm{a}_1(P), \bm{a}_2(P))$ and $\FS(\bm{a}_1(P'), \bm{a}_2(P'))$, 
where $P'=P+(0, \varepsilon)$. 
By Proposition \ref{prop:linking}, the linking number is zero. 
\end{proof}

\begin{figure}[htbp]
\centering
\begin{tikzpicture}

\draw[thin] (-5,-0.5) -- (-5, 1.5) --  (5,1.5) -- (5, -0.5);
\draw[dashed] (-5.5,0) -- (5.5,0) node[right] {$F$};

\draw[thick] (-1.5,-0.5) -- ++ (2.6,5.2) node [above] {$H_1$};
\draw[thick] (1.5,-0.5) -- ++ (-2.6,5.2) node [above] {$H_2$};
\draw[thick] (2.5,-0.5) -- ++ (-2.6,5.2) node [above] {$H_3$};

\coordinate (P1) at (0.1, 3);
\fill[red] (P1) circle (0.06) node [above] {$P_1$};
\draw[dashed, red] (P1) -- ++(-1.5, -1.5);
\draw[dashed, red] (P1) -- ++(1.5, -1.5);
\draw[red] (P1) +(-1.75, -1.75)-- ++(1.75, -1.75);
\draw[red] (P1) ++(-1.5, -1.5) rectangle ++(-0.5, -0.5);
\draw[<->, >=stealth, red] (P1) ++(-1.5, -1.5) -- ++(-0.5, -0.5);
\draw[red] (P1) ++(1.5, -1.5) rectangle ++(0.5, -0.5);
\draw[<->, >=stealth, red] (P1) ++(1.5, -1.5) -- ++(0.5, -0.5);

\coordinate (P2) at (0.4, 4);
\fill[blue] (P2) circle (0.06) node [above] {$P_2$};
\draw[dashed, blue] (P2) -- ++(-2.5, -2.5);
\draw[dashed, blue] (P2) -- ++(2.5, -2.5);
\draw[blue] (P2) +(-3.1, -3.1)-- ++(3.1, -3.1);
\draw[blue] (P2) ++(-2.5, -2.5) rectangle ++(-1.2, -1.2);
\draw[<->, >=stealth, blue] (P2) ++(-2.5, -2.5) -- ++(-1.2, -1.2);
\draw[blue] (P2) ++(2.5, -2.5) rectangle ++(1.2, -1.2);
\draw[<->, >=stealth, blue] (P2) ++(2.5, -2.5) -- ++(1.2, -1.2);

\end{tikzpicture}
\caption{Attaching circles $\FS(\bm{a}_1(P_i), \bm{a}_2(P_i))$}\label{fig:attc}
\end{figure}
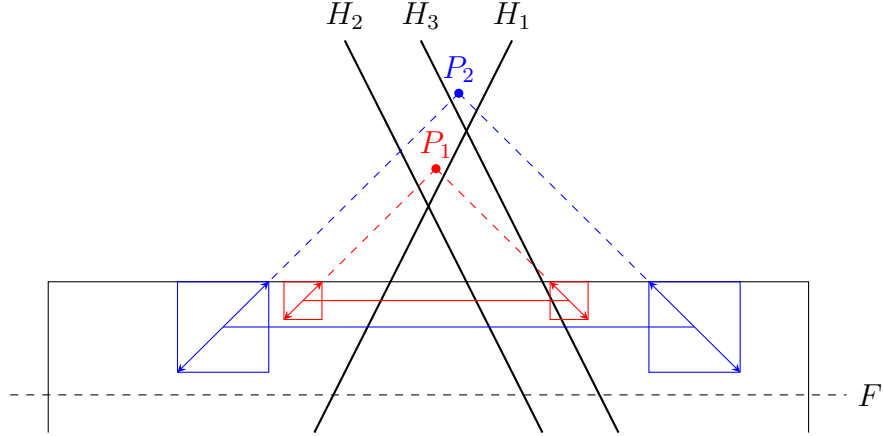

In the next section, we describe $\FS(\bm{a}_1(P), \bm{a}_2(P))$ 
as a link associated with a divide with cusps. 

\section{Kirby diagrams using divides with cusps}

\label{sec:kirby}

\subsection{Main result}

\label{subsec:main} 

Let $b=|\ch_F(\A)|$ and 
$\ch_F(\A)=\{C_1, \dots, C_b\}$. For each $1\leq s\leq b$, choose 
a point $P_s=(k_s, h_s)\in C_s$. We may assume that 
$1<h_1<\dots<h_b<R_0$. We also choose $0<h_s'<1$ and $\varepsilon>0$ 
(sufficiently small) 
such that $1>h_1'>\dots >h_b'>0$, $|h_s'-h_{s+1}'|>2\varepsilon$ 
and $\varepsilon<h_b'<h_1'<1-\varepsilon$. (E.g., $h_s'=1-\rho(h_s)$)  

We define the sign $\delta_i(C)$ of $C\in\ch_F(\A)$ with respect to 
$H_i$ as follows. 
\begin{equation}
\delta_i(C)=
\begin{cases}
+1, & \mbox{ if }\alpha_i(C)>0, \\
-1, & \mbox{ if }\alpha_i(C)<0. 
\end{cases}
\end{equation}
We note that $\delta_i(C)=+1$ (resp. $\delta_i(C)=-1$) 
is equivalent to that $C$ is sitting in the right (resp. left) side of $H_i$. 
To each $C=C_i\in\ch_F(\A)$, define a curve $\gamma(C)$ with cusps 
by joining the following parts (these parts are drawn 
in the strip $h_i'-\varepsilon<x_2<h_i'+\varepsilon$ in 
$\Rect(R, 1)$). 
\begin{itemize}
\item[(i)] 
In the left side of $\overline{H}_1$, 
if $\delta_1(C)=+1$, then  draw a leftward cusp, 
if $\delta_1(C)=-1$, then  draw a round curve as in Figure \ref{fig:left}.  
\begin{figure}[htbp]
\centering
\begin{tikzpicture}

\draw (0,0)-- node [below] {$\delta_1(C)>0$} (-4,0)--(-4,2)--(0,2);
\draw (-1.5,0)--(-0.5,2) node [above] {$\overline{H}_1$}; 
\filldraw (-0.6, 1.8) circle (0.1);
\draw[thick, blue] (0,0.5) .. controls ++(-1, 0) and ++(1,0) .. ++(-3,0.5);
\draw[thick, blue] (0,1.5) .. controls ++(-1, 0) and ++(1,0) .. ++(-3,-0.5);

\coordinate (L1) at (6,0); 
\draw (L1)-- node [below] {$\delta_1(C)<0$} ++(-4,0)--++(0,2)--++(4,0);
\draw (L1)+(-1.5,0)--++(-0.5,2) node [above] {$\overline{H}_1$}; 
\filldraw (L1)+(-0.6, 1.8) circle (0.1);
\draw[thick, blue] (L1)++(0,0.5) .. controls ++(-4, 0) and ++(-4,0) .. ++(0,1);

\end{tikzpicture}
\caption{The curve $\gamma(C)$ at the left side of $H_1$.}\label{fig:left}
\end{figure}
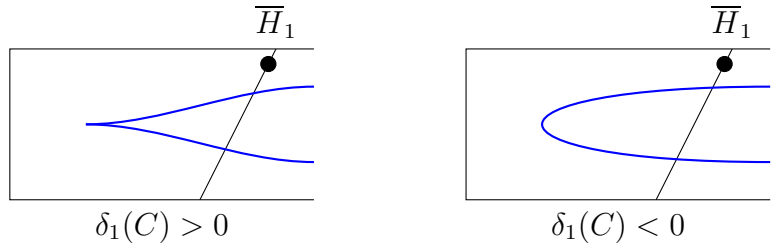
\item[(ii)] According to $(\delta_i(C), \delta_{i+1}(C))=(+1, -1), (-1, +1)$, or 
$\delta_i=\delta_{i+1}$, we draw a (two) curves between $H_i$ and $H_{i+1}$ 
as in Figure \ref{fig:between}. 
\begin{figure}[htbp]
\centering
\begin{tikzpicture}

\coordinate (B1) at (0,0); 
\coordinate (B2) at (5,0); 
\coordinate (B3) at (10,0); 

\draw (B1) -- node [below] {\small $\delta_i(C)=+1, \delta_{i+1}(C)=-1$} ++(4,0);
\draw (B1)++(0,2) -- ++(4,0);
\draw (B1)++(0.5, 0) -- ++(1,2) node [above] {$\overline{H}_i$}; 
\draw (B1)++(3.5, 0) -- ++(-1,2) node [above] {$\overline{H}_{i+1}$}; 
\filldraw (B1)+(1.4, 1.8) circle (0.1);
\filldraw (B1)+(2.6, 1.8) circle (0.1);
\draw[thick, blue] (B1)++(0,1.5)--++(4,0);
\draw[thick, blue] (B1)++(0,0.5) .. controls ++(1,0) and ++(0,-0.3) .. ++(2, 0.3);
\draw[thick, blue] (B1)++(4,0.5) .. controls ++(-1,0) and ++(0,-0.3) .. ++(-2, 0.3);

\draw (B2) -- node [below] {\small $\delta_i(C)=-1, \delta_{i+1}(C)=+1$} ++(4,0);
\draw (B2)++(0,2) -- ++(4,0);
\draw (B2)++(0.5, 0) -- ++(1,2) node [above] {$\overline{H}_i$}; 
\draw (B2)++(3.5, 0) -- ++(-1,2) node [above] {$\overline{H}_{i+1}$}; 
\filldraw (B2)+(1.4, 1.8) circle (0.1);
\filldraw (B2)+(2.6, 1.8) circle (0.1);
\draw[thick, blue] (B2)++(0,1.5)--++(4,0);
\draw[thick, blue] (B2)++(0,0.5) .. controls ++(1,0) and ++(0,0.3) .. ++(2, -0.3);
\draw[thick, blue] (B2)++(4,0.5) .. controls ++(-1,0) and ++(0,0.3) .. ++(-2, -0.3);

\draw (B3) -- node [below] {\small $\delta_i(C)=\delta_{i+1}(C)$} ++(4,0);
\draw (B3)++(0,2) -- ++(4,0);
\draw (B3)++(0.5, 0) -- ++(1,2) node [above] {$\overline{H}_i$}; 
\draw (B3)++(3.5, 0) -- ++(-1,2) node [above] {$\overline{H}_{i+1}$}; 
\filldraw (B3)+(1.4, 1.8) circle (0.1);
\filldraw (B3)+(2.6, 1.8) circle (0.1);
\draw[thick, blue] (B3)++(0,1.5)--++(4,0);
\draw[thick, blue] (B3)++(0,0.5)--++(4,0);

\end{tikzpicture}
\caption{The curve $\gamma(C)$ between $H_i$ and $H_{i+1}$.}\label{fig:between}
\end{figure}
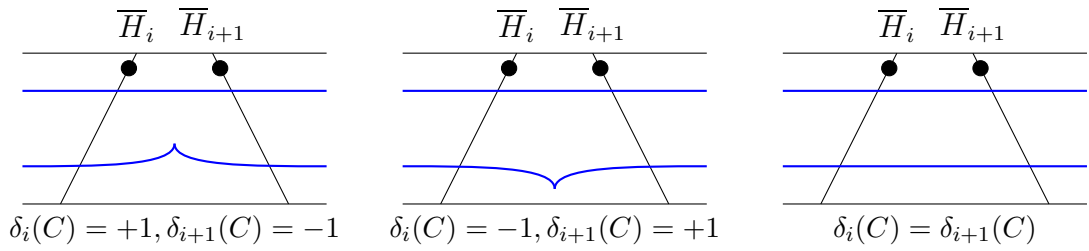
\item[(iii)] 
In the right side of $\overline{H}_n$, 
if $\delta_n(C)=+1$, then  draw a round curve, 
if $\delta_n(C)=-1$, then  draw a rightward cusp as in Figure \ref{fig:right}.  
\begin{figure}[htbp]
\centering
\begin{tikzpicture}

\coordinate (R1) at (0,0);
\coordinate (R2) at (6,0);

\draw (R1)-- node [below] {$\delta_n(C)<0$} ++(4,0)--++(0,2)--++(-4,0);
\draw (R1)++(1.5,0)--++(-1,2) node [above] {$\overline{H}_n$}; 
\filldraw (R1)+(0.6, 1.8) circle (0.1);
\draw[thick, blue] (R1)++(0,0.5) .. controls ++(1, 0) and ++(-1,0) .. ++(3,0.5);
\draw[thick, blue] (R1)++(0,1.5) .. controls ++(1, 0) and ++(-1,0) .. ++(3,-0.5);

\coordinate (R2) at (6,0); 
\draw (R2)-- node [below] {$\delta_n(C)>0$} ++(4,0)--++(0,2)--++(-4,0);
\draw (R2)++(1.5,0)--++(-1,2) node [above] {$\overline{H}_n$}; 
\filldraw (R2)+(0.6, 1.8) circle (0.1);
\draw[thick, blue] (R2)++(0,0.5) .. controls ++(4, 0) and ++(4,0) .. ++(0,1);

\end{tikzpicture}
\caption{The curve $\gamma(C)$ at the right side of $H_n$.}\label{fig:right}
\end{figure}
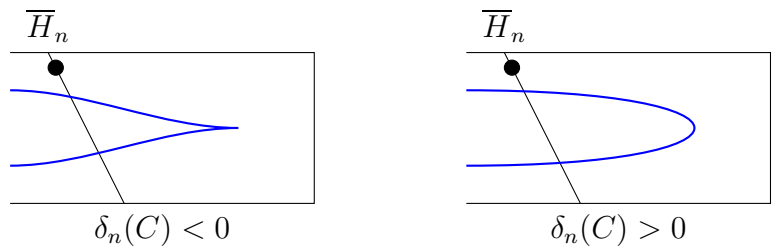
\end{itemize}

\begin{theorem}
\label{thm:main}
The link corresponding to 
$\{\gamma(C_1), \dots, \gamma(C_b)\}$ (with zero framing) present 
the attaching circles of the $2$-handles of $M(\A)$. 
\end{theorem}
We will prove Theorem \ref{thm:main} in \S \ref{sec:proof}. 

The curves $\gamma(C_1), \dots, \gamma(C_b)$ sometimes contain 
redundant cusps. Therefore, we define reduced curves 
$\gamma'(C_1), \dots, \gamma'(C_b)$ 
which are obtained from the previous curves 
by the moves formulated in Remark \ref{rem:move}. 
\begin{definition}
For a chamber $C\in\ch_F(\A)$, 
let $m_1:=\min\{i\mid\delta_i(C)=-1\}$ and 
$m_2:=\max\{i\mid\delta_i(C)=+1\}$. 
Let us define the curve $\gamma'(C)$ by joining 
above (i), (ii), (iii) for the lines 
$H_{m_1}, H_{m_1+1}, \dots, H_{m_2}$ instead of 
$H_1, \dots, H_n$. 
We call $\gamma'(C)$ the reduced curve of $\gamma(C)$. 
(Figure \ref{fig:reduce}) 
\end{definition}

\begin{proposition}
\label{prop:reduce}
The divides 
$(\gamma'(C), \overline{H}_1, \dots, \overline{H}_n)$ and 
$(\gamma(C), \overline{H}_1, \dots, \overline{H}_n)$ 
give isotopic links in $\sphere(R, 1)$. 
\end{proposition}
\begin{proof}
If $m_1=1, m_2=n$, then $\gamma'(C)=\gamma(C)$. 
Suppose $m_1>1$ (or $m_2<n$). Then $\delta_1(C)=+1$. Hence at the 
left side of $\overline{H}_1$, $\gamma(C)$ has a leftward cusp. Using the moves in 
Remark \ref{rem:move} (Figure \ref{fig:moves}), 
$\gamma(C)$ is deformed to $\gamma'(C)$. 
\end{proof}

\begin{example}
\label{ex:reduce}
Let us consider 
$4$-lines $\A=\{H_1, H_2, H_3, H_4\}$ as in Figure \ref{fig:4lines}. 
Let $\ch_F(\A)=\{C_1, \dots, C_6\}$ and fix $P_i\in C_i$ as shown in the figure. 
Then the curves $\gamma_i:=\gamma(C_i)$ ($i=1, \dots, 6$) are 
as in Figure \ref{fig:4lines} (below).
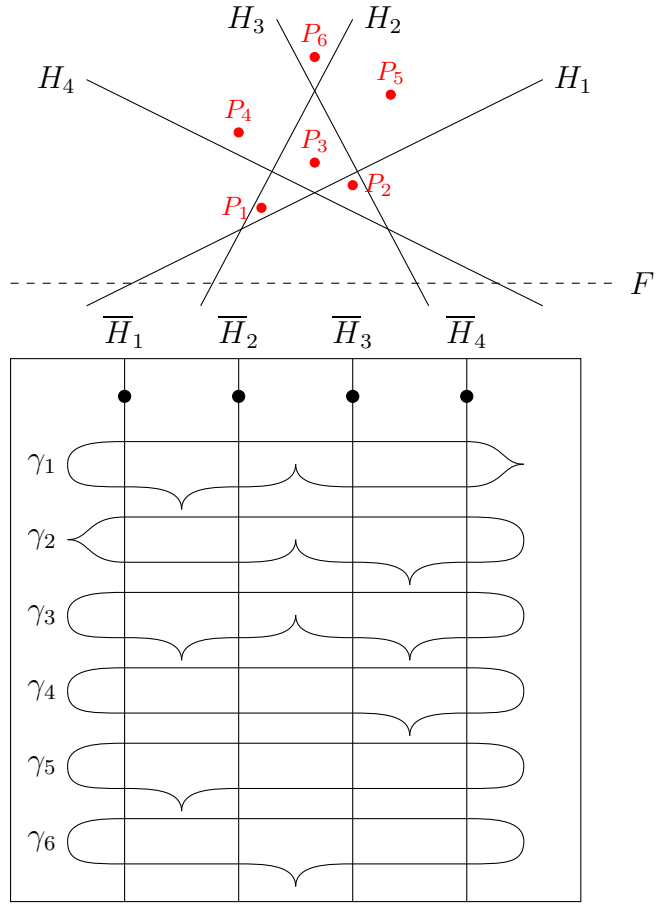
\begin{figure}[htbp]
\centering
\begin{tikzpicture}

\coordinate (O1) at (4, 8);
\coordinate (R1) at (0,0);
\coordinate (R2) at (7.5,0);
\coordinate (R3) at (0,7);
\coordinate (R4) at (7.5,7);

\draw[dashed] (O1)++(-4,0) -- ++(8,0) node [right] {$F$}; 
\draw (O1)++(-3,-0.3) -- ++(6,3) node [right] {$H_1$};
\draw (O1)++(-1.5,-0.3) -- ++(2,3.8) node [right] {$H_2$};
\draw (O1)++(1.5,-0.3) -- ++(-2,3.8) node [left] {$H_3$};
\draw (O1)++(3,-0.3) -- ++(-6,3) node [left] {$H_4$};

\filldraw[red] (O1)++(-0.7,1) circle (0.06) node [left] {\footnotesize $P_1$};
\filldraw[red] (O1)++(0.5,1.3) circle (0.06) node [right] {\footnotesize $P_2$};
\filldraw[red] (O1)++(0,1.6) circle (0.06) node [above] {\footnotesize $P_3$};
\filldraw[red] (O1)++(-1,2) circle (0.06) node [above] {\footnotesize $P_4$};
\filldraw[red] (O1)++(1,2.5) circle (0.06) node [above] {\footnotesize $P_5$};
\filldraw[red] (O1)++(0,3) circle (0.06) node [above] {\footnotesize $P_6$};

\draw (R1) +(0, -0.2) rectangle (R4); 
\draw (R1)++(1.5,-0.2) -- ++(0,7.2) node [above] {$\overline{H}_1$}; 
\filldraw (R3)++(1.5,-0.5) circle (0.08);

\draw (R1)++(3,-0.2) -- ++(0,7.2) node [above] {$\overline{H}_2$}; 
\filldraw (R3)++(3,-0.5) circle (0.08);

\draw (R1)++(4.5,-0.2) -- ++(0,7.2) node [above] {$\overline{H}_3$}; 
\filldraw (R3)++(4.5,-0.5) circle (0.08);

\draw (R1)++(6,-0.2) -- ++(0,7.2) node [above] {$\overline{H}_4$}; 
\filldraw (R3)++(6,-0.5) circle (0.08);

\draw (R1)++(0.75,5.6) node [left] {$\gamma_1$} .. controls ++(0,0.3) and ++(-0.3,0) .. ++(0.75, 0.3);
\draw (R1)++(1.5,5.9) -- ++(4.5,0);
\draw (R1)++(0.75,5.6) .. controls ++(0,-0.3) and ++(-0.5,0) .. ++(0.75, -0.3);
\draw (R1)++(1.5,5.3) .. controls ++(0.3,0) and ++(0,0.3) .. ++(0.75, -0.3);
\draw (R1)++(3,5.3) .. controls ++(-0.3,0) and ++(0,0.3) .. ++(-0.75, -0.3);
\draw (R1)++(3,5.3) .. controls ++(0.3,0) and ++(0,-0.3) .. ++(0.75, 0.3);
\draw (R1)++(4.5,5.3) .. controls ++(-0.3,0) and ++(0,-0.3) .. ++(-0.75, 0.3);
\draw (R1)++(4.5,5.3) -- ++(1.5,0);
\draw (R1)++(6,5.3) .. controls ++(0.5,0) and ++(-0.3,0) .. ++(0.75, 0.3);
\draw (R1)++(6,5.9) .. controls ++(0.5,0) and ++(-0.3,0) .. ++(0.75, -0.3);

\draw (R2)++(-0.75,4.6)  .. controls ++(0,0.3) and ++(0.3,0) .. ++(-0.75, 0.3);
\draw (R2)++(-1.5,4.9) -- ++(-4.5,0);
\draw (R2)++(-0.75,4.6) .. controls ++(0,-0.3) and ++(0.5,0) .. ++(-0.75, -0.3);
\draw (R2)++(-1.5,4.3) .. controls ++(-0.3,0) and ++(0,0.3) .. ++(-0.75, -0.3);
\draw (R2)++(-3,4.3) .. controls ++(0.3,0) and ++(0,0.3) .. ++(0.75, -0.3);
\draw (R2)++(-3,4.3) .. controls ++(-0.3,0) and ++(0,-0.3) .. ++(-0.75, 0.3);
\draw (R2)++(-4.5,4.3) .. controls ++(0.3,0) and ++(0,-0.3) .. ++(0.75, 0.3);
\draw (R2)++(-4.5,4.3) -- ++(-1.5,0);
\draw (R2)++(-6,4.3) .. controls ++(-0.5,0) and ++(0.3,0) .. ++(-0.75, 0.3);
\draw (R2)++(-6,4.9) .. controls ++(-0.5,0) and ++(0.3,0) .. ++(-0.75, -0.3) node [left] {$\gamma_2$};

\draw (R1)++(0.75,3.6) node [left] {$\gamma_3$} .. controls ++(0,0.3) and ++(-0.3,0) .. ++(0.75, 0.3);
\draw (R1)++(1.5,3.9) -- ++(4.5,0);
\draw (R1)++(0.75,3.6) .. controls ++(0,-0.3) and ++(-0.5,0) .. ++(0.75, -0.3);
\draw (R1)++(1.5,3.3) .. controls ++(0.3,0) and ++(0,0.3) .. ++(0.75, -0.3);
\draw (R1)++(3,3.3) .. controls ++(-0.3,0) and ++(0,0.3) .. ++(-0.75, -0.3);
\draw (R1)++(3,3.3) .. controls ++(0.3,0) and ++(0,-0.3) .. ++(0.75, 0.3);
\draw (R1)++(4.5,3.3) .. controls ++(-0.3,0) and ++(0,-0.3) .. ++(-0.75, 0.3);
\draw (R1)++(4.5,3.3) .. controls ++(0.3,0) and ++(0,0.3) .. ++(0.75, -0.3);
\draw (R1)++(6,3.3) .. controls ++(-0.3,0) and ++(0,0.3) .. ++(-0.75, -0.3);
\draw (R1)++(6.75,3.6) .. controls ++(0,0.3) and ++(0.3,0) .. ++(-0.75, 0.3);
\draw (R1)++(6.75,3.6) .. controls ++(0,-0.3) and ++(0.5,0) .. ++(-0.75, -0.3);

\draw (R1)++(0.75,2.6) node [left] {$\gamma_4$} .. controls ++(0,0.3) and ++(-0.3,0) .. ++(0.75, 0.3);
\draw (R1)++(1.5,2.9) -- ++(4.5,0);
\draw (R1)++(0.75,2.6) .. controls ++(0,-0.3) and ++(-0.5,0) .. ++(0.75, -0.3);
\draw (R1)++(1.5,2.3)--++(3,0);
\draw (R1)++(4.5,2.3) .. controls ++(0.3,0) and ++(0,0.3) .. ++(0.75, -0.3);
\draw (R1)++(6,2.3) .. controls ++(-0.3,0) and ++(0,0.3) .. ++(-0.75, -0.3);
\draw (R1)++(6.75,2.6) .. controls ++(0,0.3) and ++(0.3,0) .. ++(-0.75, 0.3);
\draw (R1)++(6.75,2.6) .. controls ++(0,-0.3) and ++(0.5,0) .. ++(-0.75, -0.3);

\draw (R1)++(0.75,1.6) node [left] {$\gamma_5$} .. controls ++(0,0.3) and ++(-0.3,0) .. ++(0.75, 0.3);
\draw (R1)++(1.5,1.9) -- ++(4.5,0);
\draw (R1)++(0.75,1.6) .. controls ++(0,-0.3) and ++(-0.5,0) .. ++(0.75, -0.3);
\draw (R1)++(1.5,1.3) .. controls ++(0.3,0) and ++(0,0.3) .. ++(0.75, -0.3);
\draw (R1)++(3,1.3) .. controls ++(-0.3,0) and ++(0,0.3) .. ++(-0.75, -0.3);
\draw (R1)++(3,1.3)--++(3,0);
\draw (R1)++(6.75,1.6) .. controls ++(0,0.3) and ++(0.3,0) .. ++(-0.75, 0.3);
\draw (R1)++(6.75,1.6) .. controls ++(0,-0.3) and ++(0.5,0) .. ++(-0.75, -0.3);

\draw (R1)++(0.75,0.6) node [left] {$\gamma_6$} .. controls ++(0,0.3) and ++(-0.3,0) .. ++(0.75, 0.3);
\draw (R1)++(1.5,0.9) -- ++(4.5,0);
\draw (R1)++(0.75,0.6) .. controls ++(0,-0.3) and ++(-0.5,0) .. ++(0.75, -0.3);
\draw (R1)++(3,0.3) .. controls ++(0.3,0) and ++(0,0.3) .. ++(0.75, -0.3);
\draw (R1)++(4.5,0.3) .. controls ++(-0.3,0) and ++(0,0.3) .. ++(-0.75, -0.3);
\draw (R1)++(1.5,0.3) -- ++ (1.5,0);
\draw (R1)++(4.5,0.3) -- ++ (1.5,0);
\draw (R1)++(6.75,0.6) .. controls ++(0,0.3) and ++(0.3,0) .. ++(-0.75, 0.3);
\draw (R1)++(6.75,0.6) .. controls ++(0,-0.3) and ++(0.5,0) .. ++(-0.75, -0.3);

\end{tikzpicture}
\caption{Generic four lines and its divide with cusps}\label{fig:4lines}
\end{figure}
The curves $\gamma_i, i\geq 3$, are reduced. However, $\gamma_1$ and 
$\gamma_2$ have reductions. The reduced ones $\gamma_1', \gamma_2'$ 
are as in Figure \ref{fig:reduce}. 
\begin{figure}[htbp]
\centering
\begin{tikzpicture}

\coordinate (R1) at (0,0);
\coordinate (R2) at (7.5,0);
\coordinate (R3) at (0,2);
\coordinate (R4) at (7.5,2);

\draw (R1) rectangle (R4); 
\draw (R1)++(1.5,0) -- ++(0,2); 
\filldraw (R3)++(1.5,-0.5) circle (0.08);

\draw (R1)++(3,0) -- ++(0,2); 
\filldraw (R3)++(3,-0.5) circle (0.08);

\draw (R1)++(4.5,0) -- ++(0,2); 
\filldraw (R3)++(4.5,-0.5) circle (0.08);

\draw (R1)++(6,0) -- ++(0,2); 
\filldraw (R3)++(6,-0.5) circle (0.08);

\draw (R1)++(0.75,0.9) node [left] {$\gamma_1'$} .. controls ++(0,0.3) and ++(-0.3,0) .. ++(0.75, 0.3);
\draw (R1)++(1.5,1.2) -- ++(1.5,0);
\draw (R1)++(0.75,0.9) .. controls ++(0,-0.3) and ++(-0.5,0) .. ++(0.75, -0.3);
\draw (R1)++(1.5,0.6) .. controls ++(0.3,0) and ++(0,0.3) .. ++(0.75, -0.3);
\draw (R1)++(3,0.6) .. controls ++(-0.3,0) and ++(0,0.3) .. ++(-0.75, -0.3);
\draw (R1)++(3,0.6) .. controls ++(0.3,0) and ++(0,-0.3) .. ++(0.65, 0.3);
\draw (R1)++(3,1.2) .. controls ++(0.3,0) and ++(0,0.3,0) .. ++(0.65, -0.3);

\draw (R2)++(-0.75,0.9) node [right] {$\gamma_2'$} .. controls ++(0,0.3) and ++(0.3,0) .. ++(-0.75, 0.3);
\draw (R2)++(-1.5,1.2) -- ++(-1.5,0);
\draw (R2)++(-0.75,0.9) .. controls ++(0,-0.3) and ++(0.5,0) .. ++(-0.75, -0.3);
\draw (R2)++(-1.5,0.6) .. controls ++(-0.3,0) and ++(0,0.3) .. ++(-0.75, -0.3);
\draw (R2)++(-3,0.6) .. controls ++(0.3,0) and ++(0,0.3) .. ++(0.75, -0.3);
\draw (R2)++(-3,0.6) .. controls ++(-0.3,0) and ++(0,-0.3) .. ++(-0.65, 0.3);
\draw (R2)++(-3,1.2) .. controls ++(-0.3,0) and ++(0,0.3,0) .. ++(-0.65, -0.3);


\end{tikzpicture}
\caption{Reduction of $\gamma_1$ and $\gamma_2$ of Figure \ref{fig:4lines}}\label{fig:reduce}
\end{figure}
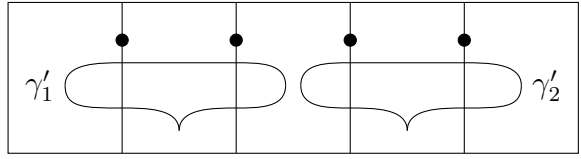
\end{example}


\subsection{Proof of Theorem \ref{thm:main}}
\label{sec:proof}

In this section, we prove Theorem \ref{thm:main}. 
Let $P\in C$. 
Here we will show 
that the link $\FS(\bm{a}_1(P), \bm{a}_2(P))$ is isotopic to 
the link determined by the curve constructed as in \S \ref{subsec:main} 
between $\bm{a}_1(P)$ and $\bm{a}_2(P)$. 
The latter can be deformed to $\gamma'(C)$ (and to $\gamma(C)$) 
by the moves in Remark \ref{rem:move} (Figure \ref{fig:compare}).  

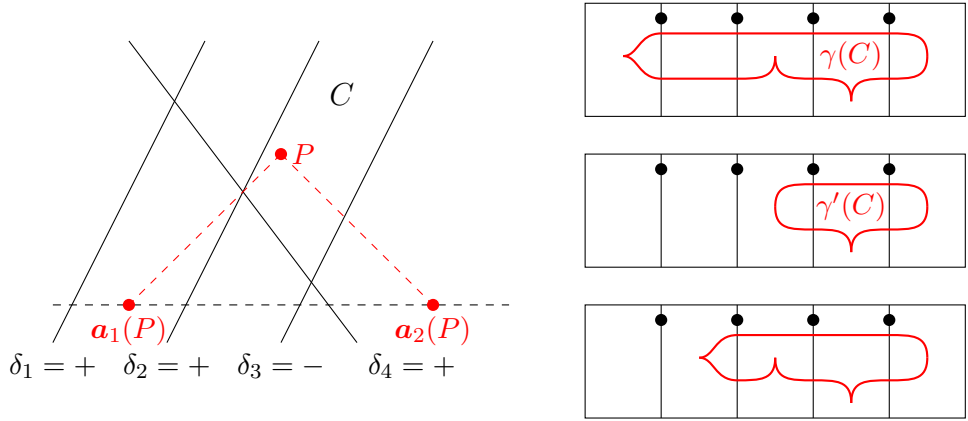
\begin{figure}[htbp]
\centering
\begin{tikzpicture}

\coordinate (Z1) at (0,1.5); 
\coordinate (G1) at (7,4); 
\coordinate (G2) at (7,2); 
\coordinate (G3) at (7,0); 

\draw[dashed] (Z1)--++(6,0);
\draw (Z1)++(0,-0.5) node [below] {\small $\delta_1=+$} --++(2,4);
\draw (Z1)++(1.5,-0.5) node [below] {\small $\delta_2=+$} --++(2,4);
\draw (Z1)++(3,-0.5) node [below] {\small $\delta_3=-$} --++(2,4);
\draw (Z1)++(4,-0.5) node [below right] {\small $\delta_4=+$} --++(-3,4);
\draw (Z1)++(3.8,2.5) node [above] {$C$};
\fill[red] (Z1)++(3,2) circle (0.08) node [right] {\small $P$};
\draw[dashed, red] (Z1)++(3,2) -- ++(-2,-2);
\draw[dashed, red] (Z1)++(3,2) -- ++(2,-2);
\fill[red] (Z1)++(1,0) circle (0.08) node [below] {\small $\bm{a}_1(P)$};
\fill[red] (Z1)++(5,0) circle (0.08) node [below] {\small $\bm{a}_2(P)$};

\draw (G1) rectangle ++(5,1.5);
\draw (G1)++(1,0) -- ++(0,1.5);
\draw (G1)++(2,0) -- ++(0,1.5);
\draw (G1)++(3,0) -- ++(0,1.5);
\draw (G1)++(4,0) -- ++(0,1.5);
\fill (G1)++(1,1.3) circle (0.08);
\fill (G1)++(2,1.3) circle (0.08);
\fill (G1)++(3,1.3) circle (0.08);
\fill (G1)++(4,1.3) circle (0.08);

\draw[red, thick] (G1)++(1,1.1)--++(3,0);
\draw[red, thick] (G1)++(0.5,0.8) .. controls ++(0.2,0) and ++(-0.3,0) .. ++(0.5,0.3);
\draw[red, thick] (G1)++(0.5,0.8) .. controls ++(0.2,0) and ++(-0.3,0) .. ++(0.5,-0.3);
\draw[red, thick] (G1)++(1,0.5)--++(1,0);
\draw[red, thick] (G1)++(2,0.5) .. controls ++(0.3,0) and ++(0,-0.3) .. ++(0.5,0.3);
\draw[red, thick] (G1)++(3,0.5) .. controls ++(-0.3,0) and ++(0,-0.3) .. ++(-0.5,0.3);
\draw[red, thick] (G1)++(3,0.5) .. controls ++(0.3,0) and ++(0,0.3) .. ++(0.5,-0.3);
\draw[red, thick] (G1)++(4,0.5) .. controls ++(-0.3,0) and ++(0,0.3) .. ++(-0.5,-0.3);
\draw[red, thick] (G1)++(3.5,0.8) node {\small $\gamma(C)$}; 
\draw[red, thick] (G1)++(4.5, 0.8) .. controls ++(0, -0.3) and ++(0.3,0) .. ++(-0.5,-0.3); 
\draw[red, thick] (G1)++(4.5, 0.8) .. controls ++(0, 0.3) and ++(0.3,0) .. ++(-0.5,0.3); 

\draw (G2) rectangle ++(5,1.5);
\draw (G2)++(1,0) -- ++(0,1.5);
\draw (G2)++(2,0) -- ++(0,1.5);
\draw (G2)++(3,0) -- ++(0,1.5);
\draw (G2)++(4,0) -- ++(0,1.5);
\fill (G2)++(1,1.3) circle (0.08);
\fill (G2)++(2,1.3) circle (0.08);
\fill (G2)++(3,1.3) circle (0.08);
\fill (G2)++(4,1.3) circle (0.08);

\draw[red, thick] (G2)++(3,1.1)--++(1,0);
\draw[red, thick] (G2)++(2.5, 0.8) .. controls ++(0, -0.3) and ++(-0.3,0) .. ++(0.5,-0.3); 
\draw[red, thick] (G2)++(2.5, 0.8) .. controls ++(0, 0.3) and ++(-0.3,0) .. ++(0.5,0.3); 
\draw[red, thick] (G2)++(3,0.5) .. controls ++(0.3,0) and ++(0,0.3) .. ++(0.5,-0.3);
\draw[red, thick] (G2)++(4,0.5) .. controls ++(-0.3,0) and ++(0,0.3) .. ++(-0.5,-0.3);
\draw[red, thick] (G2)++(3.5,0.8) node {\small $\gamma'(C)$}; 
\draw[red, thick] (G2)++(4.5, 0.8) .. controls ++(0, -0.3) and ++(0.3,0) .. ++(-0.5,-0.3); 
\draw[red, thick] (G2)++(4.5, 0.8) .. controls ++(0, 0.3) and ++(0.3,0) .. ++(-0.5,0.3);

\draw (G3) rectangle ++(5,1.5);
\draw (G3)++(1,0) -- ++(0,1.5);
\draw (G3)++(2,0) -- ++(0,1.5);
\draw (G3)++(3,0) -- ++(0,1.5);
\draw (G3)++(4,0) -- ++(0,1.5);
\fill (G3)++(1,1.3) circle (0.08);
\fill (G3)++(2,1.3) circle (0.08);
\fill (G3)++(3,1.3) circle (0.08);
\fill (G3)++(4,1.3) circle (0.08);

\draw[red, thick] (G3)++(2,1.1)--++(2,0);
\draw[red, thick] (G3)++(1.5,0.8) .. controls ++(0.2,0) and ++(-0.3,0) .. ++(0.5,0.3);
\draw[red, thick] (G3)++(1.5,0.8) .. controls ++(0.2,0) and ++(-0.3,0) .. ++(0.5,-0.3);
\draw[red, thick] (G3)++(2,0.5) .. controls ++(0.3,0) and ++(0,-0.3) .. ++(0.5,0.3);
\draw[red, thick] (G3)++(3,0.5) .. controls ++(-0.3,0) and ++(0,-0.3) .. ++(-0.5,0.3);
\draw[red, thick] (G3)++(3,0.5) .. controls ++(0.3,0) and ++(0,0.3) .. ++(0.5,-0.3);
\draw[red, thick] (G3)++(4,0.5) .. controls ++(-0.3,0) and ++(0,0.3) .. ++(-0.5,-0.3);
\draw[red, thick] (G3)++(4.5, 0.8) .. controls ++(0, -0.3) and ++(0.3,0) .. ++(-0.5,-0.3); 
\draw[red, thick] (G3)++(4.5, 0.8) .. controls ++(0, 0.3) and ++(0.3,0) .. ++(-0.5,0.3);

\end{tikzpicture}
\caption{$\gamma(C), \gamma'(C)$, and 
the curve obtained from $\FS(\bm{a}_1(P), \bm{a}_2(P))$ by the 
operation below (which is in between $\gamma(C)$ and $\gamma'(C)$). 
}\label{fig:compare}
\end{figure}

First, we consider the image of $R_2$ and $R_4$ in 
$\FS(\bm{a}_1(P), \bm{a}_2(P))$. The argument $\theta$ of 
tangent vectors satisfies 
$-\frac{\pi}{4}\leq\theta\leq\frac{\pi}{4}$ (for $R_2$)  or 
$\frac{3\pi}{4}\leq\theta\leq\frac{5\pi}{4}$ (for $R_4$). 
Hence by the assumptions on $\A$ (\S \ref{sec:setting}),  
the image of $R_2$ is isotopic to the constant vectors with 
argument $\theta=0$, and that of $R_4$ is isotopic to $\theta=\pi$. 
Therefore, the image of $R_2\cup R_4$ is isotopic to 
a part of a simple closed curve
corresponding to the tangent vectors of the upper line of $\gamma(C)$. 


Next, we consider the marginal case (the case corresponding to 
Figure \ref{fig:leftside}). Suppose that $\bm{a}_1(P)$ is in between 
$H_{i-1}$ and $H_i$. Let us consider $\sigma(R_4)$ and $\sigma(R_1)$. 
We shift $\sigma(R_4)$ upward by $\varepsilon$, and 
$\sigma(R_1)$ downward by $\varepsilon$. By connecting 
two left most points by a segment with tangent vectors of 
argument $\frac{5\pi}{4}$, we obtain a piecewise linear curve 
$c:[0, 1]\longrightarrow\Rect(R, 1)\cap H_1^{-}$ (red segments 
in Figure \ref{fig:left}) and tangent vectors 
$\bm{v}(t)\in T_{\bm{c}(t)}\R^2$.  
Note that $\{(\bm{c}(t), \bm{v}(t))\mid 0\leq t\leq 1\}$ forms the 
($\varepsilon$-shifted) image $\sigma(R_4)\cup\sigma(R_1)$.  
Now suppose that $\delta_i(C)=-1$. Then $\bm{v}(0)$ is directing 
the negative half space $H_i^{-}$ and $\bm{v}(1)$ is directing 
the positive half space $H_i^{+}$. 

There exist a curve $\widetilde{\bm{c}}(t)$ 
(the blue curve in Figure \ref{fig:leftside}) and tangent vector 
$\widetilde{\bm{v}}(t)\in T_{\widetilde{\bm{c}}(t)}\R^2$ such that 
\begin{itemize}
\item 
The curve $\widetilde{C}=\{\widetilde{\bm{c}}(t)\mid t\in [0, 1]\}$ 
is smooth. 
\item 
$\widetilde{\bm{c}}(0)=\bm{c}(0), \widetilde{\bm{c}}(1)=\bm{c}(1)$.  
\item 
$\widetilde{\bm{v}}(t)\in T_{\widetilde{\bm{c}}(t)}\widetilde{C}, 
\widetilde{\bm{v}}(t)\neq 0$. 
\item 
Let $\widetilde{\bm{v}}(t)=(\widetilde{v}_1(t), \widetilde{v}_2(t))$. Then 
$\widetilde{v}_2(0)=\widetilde{v}_2(1)=0$, $v_1(0)<0, v_1(1)>0$, 
and $\widetilde{v}_2(t)>0$ for $0<t<1$. 
\end{itemize}
Then we have 
\begin{itemize}
\item both ${\bm{v}}(0)$ and $\widetilde{\bm{v}}(0)$ are directing negative side 
of $H_i$, 
\item both ${\bm{v}}(1)$ and $\widetilde{\bm{v}}(1)$ are directing the positive side 
of $H_{i}$, and 
\item for $0<t<1$, $v_2(t)<0$ and $\widetilde{v}_2(t)<0$ hold. 
\end{itemize}
Therefore, $s\cdot \widetilde{\bm{v}}(t)+(1-s)\cdot \bm{v}(t)$ is nonzero 
for all $0\leq s\leq 1$, and 
$s\cdot (\widetilde{\bm{c}}(t), \widetilde{\bm{v}}(t))+(1-s)\cdot (\bm{c}(t), \bm{v}(t))$ 
($0\leq s\leq 1$) 
gives an isotopy between 
$(\widetilde{\bm{c}}(t), \widetilde{\bm{v}}(t))$ and  $(\bm{c}(t), \bm{v}(t))$. 
The other case ($\delta_i(C)=+1$) is similar. 

\begin{figure}[htbp]
\centering
\begin{tikzpicture}

\draw[thick] (0,0) rectangle (12,4);

\coordinate (P0) at (10,0);
\coordinate (P1) at (10.25,1);
\coordinate (P2) at (10.5,2);
\coordinate (P3) at (10.75,3);
\coordinate (P4) at (11,4);

\fill (2,2) circle (0.08) node [left] {$\bm{a}_1(P)$};

\draw[thick] (P0) -- (P4) node [above] {$\overline{H}_{i}$};
\filldraw (10.875, 3.5) circle (0.1);


\draw[red] (P3) -- node [above] {$(\bm{c}(t), \bm{v}(t))$} (2,3) -- (2,1) -- (P1);

\draw[->, >=stealth, thick, red] (P3) -- ++(-0.8, -0.3); 
\draw[->, >=stealth, thick, red] (P3) ++(-6,0) -- ++(-0.7, -0.4); 
\draw[->, >=stealth, thick, red] (2,3) -- ++(-0.6, -0.6); 
\draw[->, >=stealth, thick, red] (2,1) -- ++(-0.6, -0.6); 
\draw[->, >=stealth, thick, red] (P1) ++(-5,0) -- ++(0, -0.8); 
\draw[->, >=stealth, thick, red] (P1) -- ++(0.2, -0.7); 

\draw[blue] (P3)++(0, -0.02) .. controls ++(-2, 0) and ++(0, 1) .. (4, 2);  
\draw[blue] (P1)++(0, 0.02) .. controls ++(-2, 0) and ++(0, -1) .. (4, 2);  
\draw[->, >=stealth, thick, blue] (P3)++(0,0.05) -- ++(-1, 0); 
\draw[->, >=stealth, thick, blue] (4,2) node [right]  {$(\widetilde{\bm{c}}(t), \widetilde{\bm{v}}(t))$}-- ++(0, -0.8); 
\draw[->, >=stealth, thick, blue] (P1) -- ++(1, 0); 




\end{tikzpicture}
\caption{Isotopy between $(\widetilde{\bm{c}}(t), \widetilde{\bm{v}}(t))$ and  $(\bm{c}(t), \bm{v}(t))$}\label{fig:leftside}
\end{figure}
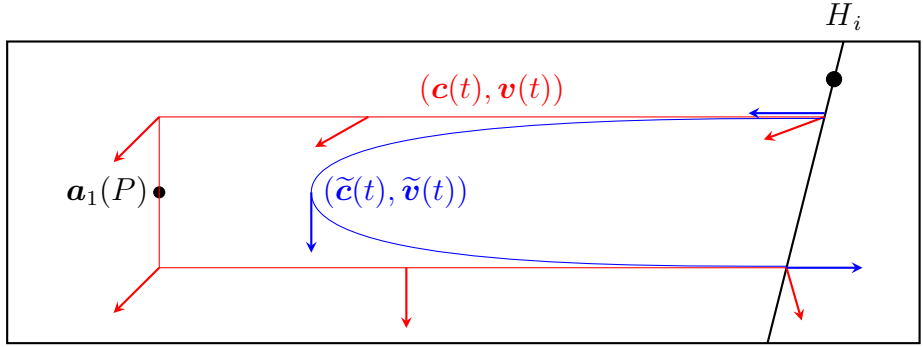

Next, we consider the case $\delta_i(C)=+1, \delta_{i+1}(C)=-1$. 
The intersection 
$[\bm{a}_1(P), \bm{a}_2(P)]\cap \{\alpha_i\geq 0, \alpha_{i+1}\leq 0\}$ 
is an interval. 
Fix a linear parametrization $c:[0, 1]\stackrel{\simeq}{\longrightarrow} 
[\bm{a}_1(P), \bm{a}_2(P)]\cap \{\alpha_i\geq 0, \alpha_{i+1}\leq 0\}$. 
Then the image of $\sigma(R_3)$ is expressed as 
$\{(\bm{c}(t), \bm{v}(t))\mid t\in[0, 1]\}$, where $\bm{v}(t)\in T_{\bm{c}(t)}\R^2$ 
which points to the fixed point $P\in C$ (i.e., $\bm{v}(t)$ is a positive scalar 
multiple of the vector $\overrightarrow{\bm{c}(t)\cdot P}$). 
In particular, since $\delta_i(C)=+1$, 
$\bm{v}(0)\in T_{c(0)}\R^2$ is directing to the positive side of $H_i$. 
Similarly, 
$\bm{v}(1)\in T_{c(1)}\R^2$ is directing to the negative side of $H_{i+1}$. 
Let $\bm{v}(t)=(v_1(t), v_2(t))$. 
Then we also have $v_2(t)>0$ for all $t\in[0, 1]$. 
There exists a curve $\widetilde{\bm{c}}(t)$ and tangent vector 
$\widetilde{\bm{v}}(t)\in T_{\bm{c}(t)}\R^2$ such that 
(Figure \ref{fig:isotopy}) 
\begin{itemize}
\item 
The curve $\widetilde{C}=\{\widetilde{\bm{c}}(t)\mid t\in [0, 1]\}$ 
has a unique upward cusp. 
\item 
$\widetilde{\bm{c}}(0)=\bm{c}(0), \widetilde{\bm{c}}(1)=\bm{c}(1)$.  
\item 
$\widetilde{\bm{v}}(t)\in T_{\widetilde{\bm{c}}(t)}\widetilde{C}$. 
\item 
Let $\widetilde{\bm{v}}(t)=(\widetilde{v}_1(t), \widetilde{v}_2(t))$. Then 
$\widetilde{v}_2(0)=\widetilde{v}_2(1)=0$, and $\widetilde{v}_2(t)>0$ for 
$0<t<1$. 
\end{itemize}
Then we have 
\begin{itemize}
\item both ${\bm{v}}(0)$ and $\widetilde{\bm{v}}(0)$ are directing the positive side 
of $H_i$, 
\item both ${\bm{v}}(1)$ and $\widetilde{\bm{v}}(1)$ are directing negative side 
of $H_{i+1}$, and 
\item for $0<t<1$, $v_2(t)>0$ and $\widetilde{v}_2(t)>0$ hold. 
\end{itemize}
Therefore, $s\cdot \widetilde{\bm{v}}(t)+(1-s)\cdot \bm{v}(t)$ is nonzero 
for all $0\leq s\leq 1$, and 
$s\cdot (\widetilde{\bm{c}}(t), \widetilde{\bm{v}}(t))+(1-s)\cdot (\bm{c}(t), \bm{v}(t))$ 
($0\leq s\leq 1$) 
gives an isotopy between 
$(\widetilde{\bm{c}}(t), \widetilde{\bm{v}}(t))$ and  $(\bm{c}(t), \bm{v}(t))$. 
Other cases ($(\delta_i(C), \delta_{i+1}(C))=(-1, +1), (+1, +1), (-1, -1)$) are similar.  
This completes the proof of Theorem \ref{thm:main}. 

\begin{figure}[htbp]
\centering
\begin{tikzpicture}

\draw[thick] (0,0) rectangle (12,3);
\draw[thick] (1,0) -- ++(0.9, 3) node [above] {$\overline{H}_i$};
\filldraw (1.81, 2.7) circle (0.1);

\draw[thick] (11,0) -- ++(-0.9,3) node [above] {$\overline{H}_{i+1}$};
\filldraw (10.19, 2.7) circle (0.1);

\draw[red] (1.3, 1) -- (10.7, 1);
\draw[->, >=stealth, thick, red] (1.3, 1) -- ++(0.6, 0.9) node [right] {$(\bm{c}(t), \bm{v}(t))$};
\draw[->, >=stealth, thick, red] (5, 1) -- ++(0.1, 1.1);
\draw[->, >=stealth, thick, red] (8, 1) -- ++(-0.2, 1.1);
\draw[->, >=stealth, thick, red] (10.7, 1) -- ++(-0.6, 0.9);

\draw[blue] (1.3, 1) .. controls ++(1.5,0) and ++(0,-0.7) .. (6, 2);
\draw[blue] (10.7, 1) .. controls ++(-1.5,0) and ++(0,-0.7) .. (6, 2);
\draw[->, >=stealth, thick, blue] (1.3,0.95) -- ++(1.2, 0) node [below] {$(\widetilde{\bm{c}}(t), \widetilde{\bm{v}}(t))$};
\draw[->, >=stealth, thick, blue] (6,2) -- ++(0, 0.8);
\draw[->, >=stealth, thick, blue] (10.7,0.95) -- ++(-1.2, 0);

\end{tikzpicture}
\caption{Isotopy between $(\widetilde{\bm{c}}(t), \widetilde{\bm{v}}(t))$ and  $(\bm{c}(t), \bm{v}(t))$}\label{fig:isotopy}
\end{figure}
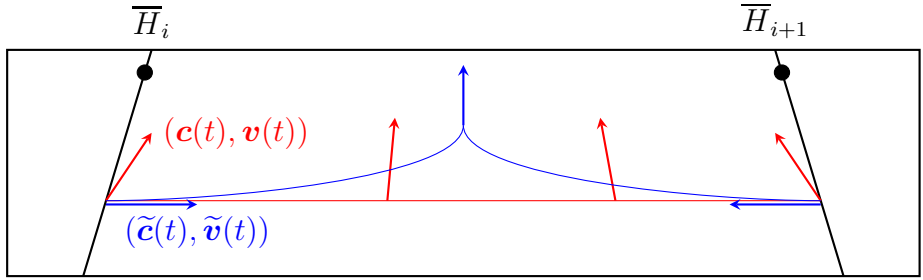

\subsection{Further examples}

Let $\Sigma_{g, n}$ be the oriented surface of genus $g$ with $n$ boundary 
components. 
Since $\Sigma_{0, 2}$ is the annulus, 
the interior of $\Sigma_{0, 2}\times\Sigma_{0, 2}$ 
is homeomorphic to $(\C^\times)^2$ which is the complement of 
the arrangement of two lines intersecting at one point. 
Hence, the minimal handle decomposition 
of $\Sigma_{0, 2}\times\Sigma_{0, 2}$ can be expressed in terms of 
divide with cusps (Figure \ref{fig:sigma0202}). 
\begin{figure}[htbp]
\centering
\begin{tikzpicture}

\draw (0,0) circle (1.5); 
\draw (0.9, 1.2) .. controls (0.45, 0.6) and (0.45, -0.6) .. (0.9, -1.2);
\filldraw (0.79, 1.04) circle (0.08);
\draw (-0.9, 1.2) .. controls (-0.45, 0.6) and (-0.45, -0.6) .. (-0.9, -1.2);
\filldraw (-0.79, 1.04) circle (0.08);

\draw (0,-1) .. controls ++(0,0.5) and ++(0,-0.5) .. ++(1,1); 
\draw (0,-1) .. controls ++(0,0.5) and ++(0,-0.5) .. ++(-1,1); 

\draw(1,0) .. controls ++(0, 0.5) and ++(0, 0.5) .. ++(-2,0) node [left] {$0$} ; 

\end{tikzpicture}
\caption{A divide with a cusp whose link is the Kirby diagram of the minimal 
handle decomposition of $\Sigma_{0, 2}\times\Sigma_{0, 2}$}\label{fig:sigma0202}
\end{figure}
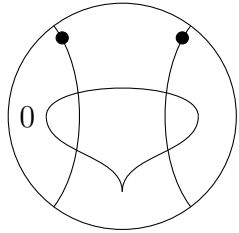
The below are related examples. 

\begin{example}
The space $\Sigma_{0, 4}\times\Sigma_{0, 2}$ is realized as the complexified 
complement of different arrangements 
$\A_1=\{\{x_1=0\}, \{x_2=0\}, \{x_1=x_2\}, \{x_1=-x_2\}\}$, and 
$\A_2=\{\{x_1=0\}, \{x_2=0\}, \{x_2=1\}, \{x_2=2\}\}$. 
(Figure \ref{fig:product})
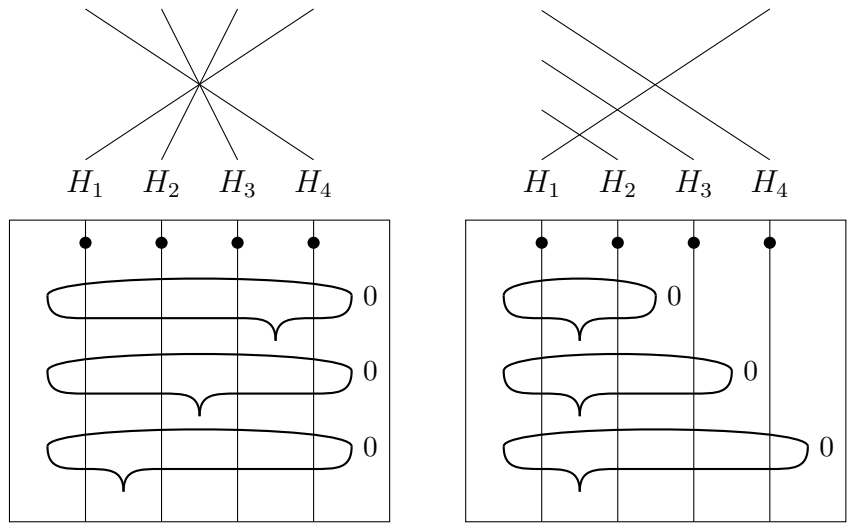
\begin{figure}[htbp]
\centering
\begin{tikzpicture}

\coordinate (A1) at (0,0);
\coordinate (A2) at (6,0);
\coordinate (O1) at (2.5, 7);
\coordinate (O2) at (8.5, 7);

\draw (A1)rectangle ++(5,4);
\draw (A1)++(1,0) -- ++(0, 4);
\draw (A1)++(2,0) -- ++(0, 4);
\draw (A1)++(3,0) -- ++(0, 4);
\draw (A1)++(4,0) -- ++(0, 4);
\fill (A1)++(1,3.7) circle (0.08);
\fill (A1)++(2,3.7) circle (0.08);
\fill (A1)++(3,3.7) circle (0.08);
\fill (A1)++(4,3.7) circle (0.08);

\draw (A1)++(1,4.8) node [below] {$H_1$} -- ++(3,2);
\draw (A1)++(2,4.8) node [below] {$H_2$} -- ++(1,2);
\draw (A1)++(3,4.8) node [below] {$H_3$} -- ++(-1,2);
\draw (A1)++(4,4.8) node [below] {$H_4$} -- ++(-3,2);

\draw[thick] (A1)++(0.5,1) .. controls ++(0, -0.3) and ++(-0.3, 0) .. ++(0.5,-0.3);
\draw[thick] (A1)++(1,0.7) .. controls ++(0.3,0) and ++(0,0.3) .. ++(0.5,-0.3);
\draw[thick] (A1)++(2,0.7) .. controls ++(-0.3,0) and ++(0,0.3) .. ++(-0.5,-0.3);
\draw[thick] (A1)++(2,0.7) -- ++(2,0);
\draw[thick] (A1)++(4.5,1) .. controls ++(0,-0.3) and ++(0.3,0) .. ++(-0.5,-0.3);
\draw[thick] (A1)++(0.5,1) .. controls ++(0, 0.3) and ++(0, 0.3) .. ++((4,0) node [right] {\small $0$};

\draw[thick] (A1)++(0.5,2) .. controls ++(0, -0.3) and ++(-0.3, 0) .. ++(0.5,-0.3);
\draw[thick] (A1)++(1,1.7) -- ++(1,0);
\draw[thick] (A1)++(2,1.7) .. controls ++(0.3,0) and ++(0,0.3) .. ++(0.5,-0.3);
\draw[thick] (A1)++(3,1.7) .. controls ++(-0.3,0) and ++(0,0.3) .. ++(-0.5,-0.3);
\draw[thick] (A1)++(3,1.7) -- ++(1,0);
\draw[thick] (A1)++(4.5,2) .. controls ++(0,-0.3) and ++(0.3,0) .. ++(-0.5,-0.3);
\draw[thick] (A1)++(0.5,2) .. controls ++(0, 0.3) and ++(0, 0.3) .. ++((4,0) node [right] {\small $0$};

\draw[thick] (A1)++(0.5,3) .. controls ++(0, -0.3) and ++(-0.3, 0) .. ++(0.5,-0.3);
\draw[thick] (A1)++(1,2.7) -- ++(2,0);
\draw[thick] (A1)++(3,2.7) .. controls ++(0.3,0) and ++(0,0.3) .. ++(0.5,-0.3);
\draw[thick] (A1)++(4,2.7) .. controls ++(-0.3,0) and ++(0,0.3) .. ++(-0.5,-0.3);
\draw[thick] (A1)++(4.5,3) .. controls ++(0,-0.3) and ++(0.3,0) .. ++(-0.5,-0.3);
\draw[thick] (A1)++(0.5,3) .. controls ++(0, 0.3) and ++(0, 0.3) .. ++((4,0) node [right] {\small $0$};

\draw (A2)rectangle ++(5,4);
\draw (A2)++(1,0) -- ++(0, 4);
\draw (A2)++(2,0) -- ++(0, 4);
\draw (A2)++(3,0) -- ++(0, 4);
\draw (A2)++(4,0) -- ++(0, 4);
\fill (A2)++(1,3.7) circle (0.08);
\fill (A2)++(2,3.7) circle (0.08);
\fill (A2)++(3,3.7) circle (0.08);
\fill (A2)++(4,3.7) circle (0.08);

\draw (A2)++(1,4.8) node [below] {$H_1$} -- ++(3,2);
\draw (A2)++(2,4.8) node [below] {$H_2$} -- ++(-1,0.66);
\draw (A2)++(3,4.8) node [below] {$H_3$} -- ++(-2,1.32);
\draw (A2)++(4,4.8) node [below] {$H_4$} -- ++(-3,1.98);

\draw[thick] (A2)++(0.5,1) .. controls ++(0, -0.3) and ++(-0.3, 0) .. ++(0.5,-0.3);
\draw[thick] (A2)++(1,0.7) .. controls ++(0.3,0) and ++(0,0.3) .. ++(0.5,-0.3);
\draw[thick] (A2)++(2,0.7) .. controls ++(-0.3,0) and ++(0,0.3) .. ++(-0.5,-0.3);
\draw[thick] (A2)++(2,0.7) -- ++(2,0);
\draw[thick] (A2)++(4.5,1) .. controls ++(0,-0.3) and ++(0.3,0) .. ++(-0.5,-0.3);
\draw[thick] (A2)++(0.5,1) .. controls ++(0, 0.3) and ++(0, 0.3) .. ++((4,0) node [right] {\small $0$};

\draw[thick] (A2)++(0.5,2) .. controls ++(0, -0.3) and ++(-0.3, 0) .. ++(0.5,-0.3);
\draw[thick] (A2)++(1,1.7) .. controls ++(0.3,0) and ++(0,0.3) .. ++(0.5,-0.3);
\draw[thick] (A2)++(2,1.7) .. controls ++(-0.3,0) and ++(0,0.3) .. ++(-0.5,-0.3);
\draw[thick] (A2)++(2,1.7) -- ++(1,0);
\draw[thick] (A2)++(3.5,2) .. controls ++(0,-0.3) and ++(0.3,0) .. ++(-0.5,-0.3);
\draw[thick] (A2)++(0.5,2) .. controls ++(0, 0.3) and ++(0, 0.3) .. ++((3,0) node [right] {\small $0$};

\draw[thick] (A2)++(0.5,3) .. controls ++(0, -0.3) and ++(-0.3, 0) .. ++(0.5,-0.3);
\draw[thick] (A2)++(1,2.7) .. controls ++(0.3,0) and ++(0,0.3) .. ++(0.5,-0.3);
\draw[thick] (A2)++(2,2.7) .. controls ++(-0.3,0) and ++(0,0.3) .. ++(-0.5,-0.3);
\draw[thick] (A2)++(2.5,3) .. controls ++(0,-0.3) and ++(0.3,0) .. ++(-0.5,-0.3);
\draw[thick] (A2)++(0.5,3) .. controls ++(0, 0.3) and ++(0, 0.3) .. ++((2,0) node [right] {\small $0$};

\end{tikzpicture}
\caption{Different divides with cusps for $\Sigma_{0, 4}\times\Sigma_{0, 2}$}\label{fig:product}
\end{figure}
\end{example}

\begin{example}
The space $\Sigma_{0, 4}\times\Sigma_{0, 3}$ is realized as the complexified 
complement of the arrangement 
$\A=\{\{x_1=0\}, \{x_1=1\}, \{x_2=0\}, \{x_2=1\}, \{x_2=2\}\}$. 
(Figure \ref{fig:product23})
\begin{figure}[htbp]
\centering
\begin{tikzpicture}

\coordinate (A2) at (0,0);

\draw (A2)++(-1,0) rectangle ++(6,7);
\draw (A2)++(0,0) -- ++(0, 7);
\draw (A2)++(1,0) -- ++(0, 7);
\draw (A2)++(2,0) -- ++(0, 7);
\draw (A2)++(3,0) -- ++(0, 7);
\draw (A2)++(4,0) -- ++(0, 7);
\fill (A2)++(0,6.7) circle (0.08);
\fill (A2)++(1,6.7) circle (0.08);
\fill (A2)++(2,6.7) circle (0.08);
\fill (A2)++(3,6.7) circle (0.08);
\fill (A2)++(4,6.7) circle (0.08);

\draw (A2)++(0,7.8) node [below] {$H_1$} -- ++(3,2);
\draw (A2)++(1,7.8) node [below] {$H_2$} -- ++(3,2);
\draw (A2)++(2,7.8) node [below] {$H_3$} -- ++(-2,1.32);
\draw (A2)++(3,7.8) node [below] {$H_4$} -- ++(-3,1.98);
\draw (A2)++(4,7.8) node [below] {$H_5$} -- ++(-3,1.98);

\draw[thick] (A2)++(0.5,4) .. controls ++(0, -0.3) and ++(-0.3, 0) .. ++(0.5,-0.3);
\draw[thick] (A2)++(1,3.7) .. controls ++(0.3,0) and ++(0,0.3) .. ++(0.5,-0.3);
\draw[thick] (A2)++(2,3.7) .. controls ++(-0.3,0) and ++(0,0.3) .. ++(-0.5,-0.3);
\draw[thick] (A2)++(2,3.7) -- ++(2,0);
\draw[thick] (A2)++(4.5,4) .. controls ++(0,-0.3) and ++(0.3,0) .. ++(-0.5,-0.3);
\draw[thick] (A2)++(0.5,4) .. controls ++(0, 0.3) and ++(0, 0.3) .. ++(4,0) node [right] {\small $0$};

\draw[thick] (A2)++(0.5,5) .. controls ++(0, -0.3) and ++(-0.3, 0) .. ++(0.5,-0.3);
\draw[thick] (A2)++(1,4.7) .. controls ++(0.3,0) and ++(0,0.3) .. ++(0.5,-0.3);
\draw[thick] (A2)++(2,4.7) .. controls ++(-0.3,0) and ++(0,0.3) .. ++(-0.5,-0.3);
\draw[thick] (A2)++(2,4.7) -- ++(1,0);
\draw[thick] (A2)++(3.5,5) .. controls ++(0,-0.3) and ++(0.3,0) .. ++(-0.5,-0.3);
\draw[thick] (A2)++(0.5,5) .. controls ++(0, 0.3) and ++(0, 0.3) .. ++(3,0) node [right] {\small $0$};

\draw[thick] (A2)++(0.5,6) .. controls ++(0, -0.3) and ++(-0.3, 0) .. ++(0.5,-0.3);
\draw[thick] (A2)++(1,5.7) .. controls ++(0.3,0) and ++(0,0.3) .. ++(0.5,-0.3);
\draw[thick] (A2)++(2,5.7) .. controls ++(-0.3,0) and ++(0,0.3) .. ++(-0.5,-0.3);
\draw[thick] (A2)++(2.5,6) .. controls ++(0,-0.3) and ++(0.3,0) .. ++(-0.5,-0.3);
\draw[thick] (A2)++(0.5,6) .. controls ++(0, 0.3) and ++(0, 0.3) .. ++(2,0) node [right] {\small $0$};

\draw[thick] (A2)++(-0.5,1) .. controls ++(0, -0.3) and ++(-0.3, 0) .. ++(0.5,-0.3);
\draw[thick] (A2)++(0,0.7) -- ++(1,0);
\draw[thick] (A2)++(1,0.7) .. controls ++(0.3,0) and ++(0,0.3) .. ++(0.5,-0.3);
\draw[thick] (A2)++(2,0.7) .. controls ++(-0.3,0) and ++(0,0.3) .. ++(-0.5,-0.3);
\draw[thick] (A2)++(2,0.7) -- ++(2,0);
\draw[thick] (A2)++(4.5,1) .. controls ++(0,-0.3) and ++(0.3,0) .. ++(-0.5,-0.3);
\draw[thick] (A2)++(-0.5,1) .. controls ++(0, 0.3) and ++(0, 0.3) .. ++(5,0) node [right] {\small $0$};

\draw[thick] (A2)++(-0.5,2) .. controls ++(0, -0.3) and ++(-0.3, 0) .. ++(0.5,-0.3);
\draw[thick] (A2)++(0,1.7) -- ++(1,0);
\draw[thick] (A2)++(1,1.7) .. controls ++(0.3,0) and ++(0,0.3) .. ++(0.5,-0.3);
\draw[thick] (A2)++(2,1.7) .. controls ++(-0.3,0) and ++(0,0.3) .. ++(-0.5,-0.3);
\draw[thick] (A2)++(2,1.7) -- ++(1,0);
\draw[thick] (A2)++(3.5,2) .. controls ++(0,-0.3) and ++(0.3,0) .. ++(-0.5,-0.3);
\draw[thick] (A2)++(-0.5,2) .. controls ++(0, 0.3) and ++(0, 0.3) .. ++(4,0) node [right] {\small $0$};

\draw[thick] (A2)++(-0.5,3) .. controls ++(0, -0.3) and ++(-0.3, 0) .. ++(0.5,-0.3);
\draw[thick] (A2)++(0,2.7) -- ++(1,0);
\draw[thick] (A2)++(1,2.7) .. controls ++(0.3,0) and ++(0,0.3) .. ++(0.5,-0.3);
\draw[thick] (A2)++(2,2.7) .. controls ++(-0.3,0) and ++(0,0.3) .. ++(-0.5,-0.3);
\draw[thick] (A2)++(2.5,3) .. controls ++(0,-0.3) and ++(0.3,0) .. ++(-0.5,-0.3);
\draw[thick] (A2)++(-0.5,3) .. controls ++(0, 0.3) and ++(0, 0.3) .. ++(3,0) node [right] {\small $0$};

\end{tikzpicture}
\caption{A divide with cusps for $\Sigma_{0, 4}\times\Sigma_{0, 3}$}\label{fig:product23}
\end{figure}
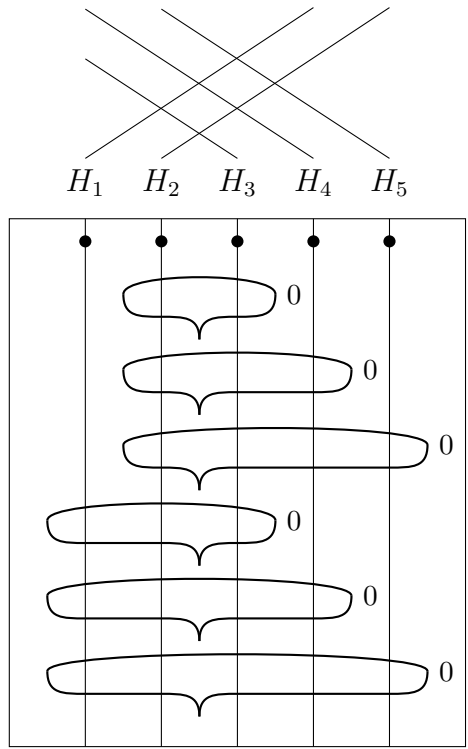
\end{example}




\medskip

\noindent
{\bf Acknowledgements.} 
Masahiko Yoshinaga 
was partially supported by JSPS KAKENHI 
Grant Numbers JP19K21826, JP18H01115. 
The authors thank 
Prof. Masaharu Ishikawa for helpful comments on the first draft. 
They also thank the referee for careful reading and comments 
to improve the paper.

\end{document}